\documentclass[10pt, a4paper, reqno]{amsart}
\usepackage{extra/style}

\begin{document}
    
    \begin{abstract}
    Drawing on the theory of Minimal Model Program singularities for foliations, we define relative canonical and log-canonical singularities for algebraic stacks with finite generic stabilisers.
    We show that if a point has log-canonical singularities, its stabiliser group is a finite extension of an algebraic torus, thus, \'etale locally, the good moduli space exists.
    If the singularity is canonical, we further show that the locus of stable points is non-empty.
\end{abstract}

\maketitle

\tableofcontents
    \section*{Introduction}
\label{sec:introduction}

There is an intimate relation between Birational Geometry and the construction of quotient spaces.
In both endeavours, one strives to find a distinguished \emph{quotient space} in a fixed birational equivalence class.
In this paper, we further investigate this relation by studying how the singularities of the Minimal Model Program affect the local construction of quotient spaces of algebraic stacks.

\begin{mainthm}
    Let $\mathcal{X}$ be a normal algebraic stack of finite type over an algebraically closed field $\mathbbm{k}$ of characteristic zero with affine diagonal and finite generic stabilisers.
    Let $x \in \mathcal{X}(\mathbbm{k})$ be a point. Then,
    \begin{enumerate}
        \item if $\mathcal{X}$ has relative log-canonical singularities at $x$, the stabiliser group of $x$ is a finite extension of an algebraic torus and a good moduli space exists étale locally on $\mathcal{X}$, and
        \item if $\mathcal{X}$ has relative canonical singularities at $x$, the locus of stable points with respect to the local good moduli space is non-empty.
    \end{enumerate}
\end{mainthm}

We say that an algebraic stack $\mathcal{X}$ has \emph{relative (log-)canonical singularities} if there exists a smooth presentation $X \rightarrow \mathcal{X}$ such that the relative tangent sheaf $\mathscr{T}_{X/\mathcal{X}}$ has (log-)canonical singularities as a foliation on $X$.
Similarly, we say $\mathcal{X}$ is \emph{relatively regular} if $\mathscr{T}_{X/\mathcal{X}}$ is a regular foliation on $X$.
Since singularities of the Minimal Model Program are insensitive to smooth morphisms, these notions are independent of the choice of presentation (see Definition \ref{def:relative_singularities}). \par

When $\mathcal{X}$ admits a good moduli space $q : \mathcal{X} \rightarrow Q$, we say that $x \in \mathcal{X}$ is \emph{stable} if $q^{-1}(q(x)) = \{x\}$ (see \cite[Definition 2.5]{MR4255045}). \par

Recall that, étale locally on $X$, a coarse moduli space exists when the stabiliser is finite (\cite[Theorem 1.1]{MR1432041}), and a good moduli space exists when the stabiliser is reductive (\cite[Theorem 1.1]{MR4088350}).
Thus our Main Theorem yields the following.

\begin{table}[h]
    \centering
    \begin{tabularx}{\textwidth}{CCC}
        \toprule
        \textbf{Type of Singularity} & \textbf{Stabiliser Group} & \textbf{Local Moduli Space} \\
        \midrule
        Regular & Finite & Coarse \\
        \addlinespace[0.5em]
        Canonical & Toric & Stable \\
        \addlinespace[0.5em]
        Log-Canonical & Toric & Good \\
        \bottomrule
    \end{tabularx}
\end{table}

We deduce two corollaries.

\begin{corollaryA}
    Let $G$ be an affine algebraic group acting with finite generic stabilisers on a normal scheme $X$ of finite type over an algebraically closed field $\mathbbm{k}$ of characteristic zero.
    Let $x \in X$ be a closed point and suppose that the foliation induced by $G$ is log-canonical at $x$.
    Then the stabiliser group of $x$ is a finite extension of an algebraic torus.
\end{corollaryA}

\begin{corollaryB}
    Let $X \rightarrow \mathcal{X}$ be a smooth presentation of a normal algebraic stack of finite type over an algebraically closed field $\mathbbm{k}$ of characteristic zero with affine diagonal and finite generic stabilisers.
    Let $x \in X$ be a closed point and suppose that the induced foliation $\mathscr{F} := \mathscr{T}_{X/\mathcal{X}}$ is canonical at $x$.
    Then there exists an étale neighbourhood $U \rightarrow X$ of $x$ and a morphism $U \rightarrow Q$ to an affine scheme $Q$ of finite type over $\mathbbm{k}$, such that $\mathscr{F}|_U = \mathscr{T}_{U/Q}$ over the generic point of $U$.
\end{corollaryB}

\subsection*{Motivation}

Our main interest is algebraic foliations.
Notably, we would like to generalise the following (\cite[Fact I.ii.4]{MR3128985}) to higher rank locally free foliations.

\begin{theorem*}
    [McQuillan--Panazzolo]
    Let $X$ be a Noetherian normal scheme over $\mathbb{C}$ and let $x \in X$ be a closed point.
    Suppose that $\partial$ is a local derivation of $X$ such that $x$ is $\partial$-invariant.
    Then $\partial$ is log-canonical at $x$ if and only if the induced endomorphism of the tangent space $T_{X, x}$ is non-nilpotent.
\end{theorem*}

In higher rank, in order to formulate a generalisation, we have to consider the stabiliser Lie algebras of the foliation and their normal representations (see \S \ref{subsec:normal_representations}).
If we then assume that $\mathscr{F}$ is induced by a smooth presentation of an algebraic stack, i.e. $\mathscr{F}$ is algebraically integrable in a strong sense, then part (1) of our Main Theorem immediately implies that $\mathscr{F}$ is étale locally induced by an algebraic torus action, thus, in particular, it is non-nilpotent.
Corollary A is an instance where we can apply our result. \par

Another significant motivating conjecture is \cite[Conjecture 6.4.2]{MR4844626}, which, from our perspective, states that canonical foliations admit a stable quotient space.

\begin{conjecture*}
    [Cascini--Spicer]
    Let $X$ be a klt $\mathbb{Q}$-factorial projective variety over $\mathbb{C}$ and let $\mathscr{F}$ be an algebraically integrable foliation on $X$ with canonical singularities.
    Then there exists a projective morphism $f : X \rightarrow Y$ to a projective variety $Y$ such that $\mathscr{F} = \mathscr{T}_f$ over the generic point of $X$.
\end{conjecture*}

If we again assume that $\mathscr{F}$ is induced by a smooth presentation of an algebraic stack, then Corollary B states that the conjecture is true étale locally on $X$.
More precisely, the foliation is étale locally induced by a stable good moduli space.
We may hope, perhaps under further global assumptions, to be able to glue these local moduli spaces, as in \cite{MR4088350}, in order to obtain a global algebraic space quotient. \par

We also want to mention that, in \cite[Theorem 2.1.9]{chen}, the authors provide an answer to the Cascini--Spicer conjecture when the foliated pair $(X, \mathscr{F})$ is \emph{F-dlt}.
Our methods are different.
In \emph{loc. cit.}, the authors resolve an indeterminacy locus (as we also do in \S \ref{subsec:indeterminacy_locus}) and then run a Minimal Model Program for foliations in order to construct the morphism $f$.
On the other hand, we show that a local good moduli space exists and prove it is of maximal dimension.

\subsection*{Key Examples}

We always keep two important examples in mind illustrating the different behaviour of canonical and log-canonical singularities. Let $X = \mathbb{A}^2$.

\noindent{
\begin{minipage}{0.48\textwidth}
    Let $\mathbb{G}_m$ act on $X$ with weights $(1, 1)$.
\end{minipage}
\hfill
\begin{minipage}{0.48\textwidth}
    Let $\mathbb{G}_m$ act on $X$ with weights $(1, -1)$.
\end{minipage}

\vspace{-0.5em}
\centering{The induced foliation $\mathscr{F}$ at a point $x_0 \in X$ is the tangent space of the orbit of $x_0$.
One can see that $\mathscr{F}$ is globally generated by the vector field\\[-0.25em]}
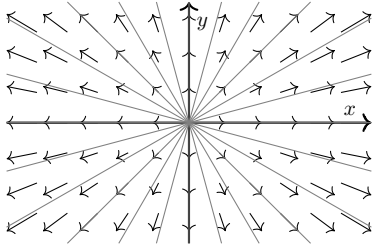
\begin{minipage}{0.48\textwidth}
    \begin{align*}
        \partial := x \partial_x + y \partial_y.
    \end{align*}
    \begin{center}
        \begin{tikzpicture}[scale=0.8, transform shape]
            \draw[thick, ->] (-3,0) -- (3,0) node[anchor=south, xshift=-10pt] {$x$};
            \draw[thick, ->] (0,-2) -- (0,2) node[anchor=west, yshift=-10pt] {$y$};
            \foreach \x in {-2.5,-2,-1.5,-1,-0.5,0.5,1,1.5,2,2.5}
                \foreach \y in {-1.5,-1,-0.5,0,0.5,1,1.5}
                    \draw[->] (\x,\y) -- (1.2*\x,1.2*\y);
            \foreach \y in {-1.5,-1,-0.5,0.5,1,1.5}
                \draw[->] (0,\y) -- (0,1.2*\y);
            \foreach \t in {-45,-30,-15,0,15,30,45}
                \draw[gray] ({-2*tan(\t)}, -2) -- ({2*tan(\t)}, 2);
            \foreach \t in {-30,-15,0,15,30}
                \draw[gray] (-3, {-3*tan(\t)}) -- (3, {3*tan(\t)});
        \end{tikzpicture}
    \end{center}
    The origin is a log-canonical singularity of $\mathscr{F}$ and all other points are regular. \par
    The generic orbit is not closed. \par
    $\mathbb{G}_m$ acts on $U := X \setminus \{0\}$ and the quotient stack $[U/\mathbb{G}_m]$ is $\mathbb{P}^1$. \par
    We obtain a rational map
    \begin{align*}
        f : X \supseteq U &\rightarrow \mathbb{P}^1 \\
        (x, y) &\dashedrightarrow (x : y).
    \end{align*}
\end{minipage}
\hfill
\begin{minipage}{0.48\textwidth}
    \begin{align*}
        \partial := x \partial_x - y \partial_y.
    \end{align*}
    \begin{center}
        \begin{tikzpicture}[scale=0.8, transform shape]
            \draw[thick, ->] (-3,0) -- (3,0) node[anchor=south, xshift=-10pt] {$x$};
            \draw[thick, ->] (0,-2) -- (0,2) node[anchor=west, yshift=-10pt] {$y$};
                \foreach \x in {-2.5,-2,-1.5,-1,-0.5,0.5,1,1.5,2,2.5}
                    \foreach \y in {-1.5,-1,-0.5,0,0.5,1,1.5}
                        \draw[->] (\x,\y) -- (1.2*\x,0.8*\y);
                \foreach \y in {-1.5,-1,-0.5,0.5,1,1.5}
                    \draw[->] (0,\y) -- (0,0.8*\y);
                \foreach \k in {-4.5,-3.5,-2.5,-1.5,-1,-0.5,-0.2}
                    \draw[gray] (0,0) plot[domain=-3:(\k/2), samples = 300] (\x,\k/\x);
                \foreach \k in {-4.5,-3.5,-2.5,-1.5,-1,-0.5,-0.2}
                    \draw[gray] (0,0) plot[domain=-3:(\k/2), samples = 300] (\x,-\k/\x);
                \foreach \k in {4.5,3.5,2.5,1.5,1,0.5,0.2}
                    \draw[gray] (0,0) plot[domain=(\k/2):3, samples = 300] (\x,\k/\x);
                \foreach \k in {4.5,3.5,2.5,1.5,1,0.5,0.2}
                    \draw[gray] (0,0) plot[domain=(\k/2):3, samples = 300] (\x,-\k/\x);
        \end{tikzpicture}
    \end{center}
    The origin is a canonical singularity of $\mathscr{F}$ and all other points are regular. \par
    The generic orbit is closed. \par
    $\mathbb{G}_m$ acts on $U := X \setminus \{xy=0\}$ and the quotient stack $[U/\mathbb{G}_m]$ is $\mathbb{A}^1 \setminus \{0\}$. \par
    We obtain a rational map
    \begin{align*}
        f : X \supseteq U &\rightarrow \mathbb{A}^1\setminus \{0\} \subseteq \mathbb{P}^1 \\
        (x, y) &\dashedrightarrow (x : 1/y).
    \end{align*}
\end{minipage}

\vspace{-0.5em}
\centering{We can resolve the indeterminacy locus of $f$ by blowing up the origin $\pi : Z \rightarrow X$.
We get a morphism $g : (Z, \Delta) \rightarrow \mathbb{P}^1$, where $\Delta \subseteq Z$ is the exceptional divisor of $\pi$.
Since the origin is a fixed point, the $\mathbb{G}_m$-action lifts to $Z$ and $g$ is $\mathbb{G}_m$-invariant.\\[0.5em]}

\begin{minipage}{0.48\textwidth}
    Note that $g(\Delta) = \mathbb{P}^1$. \par
    This shows that $g$ cannot descend to a $\mathbb{G}_m$-invariant morphism $f : X \rightarrow \mathbb{P}^1$. \par
    The good moduli space of the stack $\mathcal{X} := [X / \mathbb{G}_m]$ is $\spec \mathbbm{k}$, however $\mathcal{X}$ does not contain any stable point.
\end{minipage}
\hfill
\begin{minipage}{0.48\textwidth}
    Note that $g(\Delta)$ is a point. \par
    This shows that $g$ descends to a $\mathbb{G}_m$-invariant morphism $f : X \rightarrow \mathbb{P}^1$. \par
    The good moduli space of the stack $\mathcal{X} := [X / \mathbb{G}_m]$ is $\mathbb{A}^1$, and $\mathcal{X}$ contains an open dense subset of stable points.
\end{minipage}}

\subsection*{Main Ideas}

The main ingredient is semistable reduction.
Variants of this tool, such as \cite[Theorem 0.3]{MR1738451}, have already been successfully employed in the study of algebraically integrable foliations in \cite{ambro}.
We use \emph{functorial semistable reduction} (\cite[Theorem 1.2.17]{abramovich2020}) and proceed in several steps. \par

Let $X \rightarrow \mathcal{X}$ be a minimal presentation around a closed point $x$ and let $\mathscr{F}$ be the induced locally free foliation on $X$.

\begin{enumerate}[itemsep=1em, label=\textnormal{(\Roman*)}]
    \item We first observe that an algebraic stack $\mathcal{X}$ with finite generic stabilisers is generically a smooth quotient stack by a finite group.
    As a result, there exists a rational map $f : \mathcal{X} \dashedrightarrow B$ to a smooth projective scheme $B$ (\S \ref{subsec:rational_map}).
    
    \item We can then functorially resolve the indeterminacy locus of $f$ in order to obtain a modification $\sigma : \mathcal{Z} \rightarrow \mathcal{X}$ represented by $\sigma_X : Z \rightarrow X$, and a morphism $g : \mathcal{Z} \rightarrow B$.
    We deduce that $\sigma_X^* \mathscr{F}$ is a foliation on $Z$ (\S \ref{subsec:indeterminacy_locus}).
    
    \item We now use functorial semistable reduction to find an alteration $\tau_Z : Z^{\mathrm{ss}} \rightarrow Z$ and a semistable morphism $g^{\mathrm{ss}} : Z^{\mathrm{ss}} \rightarrow B^{\mathrm{ss}}$ such that $\mathscr{F}^{\mathrm{ss}} := \pi_X^* \mathscr{F}$ is a foliation on $Z^{\mathrm{ss}}$, where $\pi_X = \sigma_X \circ \tau_Z$ (\S \ref{subsec:semistable_reduction}).
    
    \item By construction $\mathscr{F}^{\mathrm{ss}}$ is a subsheaf of the logarithmic tangent sheaf $\mathscr{T}_{g^{\mathrm{ss}}}$ of $g^{\mathrm{ss}}$.
    We show that, when $\mathscr{F}$ is log-canonical, $\mathscr{F}^{\mathrm{ss}}$ is log-canonical (Lemma \ref{lem:log_canonical_alteration}).
    It follows that $\mathscr{F}^{\mathrm{ss}} = \mathscr{T}_{g^{\mathrm{ss}}}$, else we could extract a divisor with negative log-discrepancy (\S \ref{subsec:formal_representability}).
    
    \item Since $g^{\mathrm{ss}}$ is semistable, it induces local torus actions.
    The differentials of these actions are captured by the \emph{normal representations} of $\mathscr{T}_{g^{\mathrm{ss}}}$, which are faithful and semisimple representations of Abelian Lie algebras (\S \ref{subsec:semistable_morphisms}).
    
    \item Now suppose $G$ is the stabiliser group of $x$.
    Then $G$ acts on a proper space $V := \pi^{-1}(x) \subseteq Z^{\mathrm{ss}}$.
    Suppose it has a fixed point $v \in V$, then, infinitesimally around $v$, $G$ acts as $\mathscr{F}^{\mathrm{ss}}$.
    But $\mathscr{F}^{\mathrm{ss}} = \mathscr{T}_{g^{\mathrm{ss}}}$ (Step (IV)), and $\mathscr{T}_{g^{\mathrm{ss}}}$ acts as a torus (Step (V)).
    We deduce that $G$ must be a torus (\S \ref{subsec:main_log_canonical}).

    \item If $\mathscr{F}$ is canonical at $x$, we show that the image of the non-regular locus of $\sigma_X^* \mathscr{F}$ on $Z$ via the morphism $g$ is not dense (Proposition \ref{prop:canonical_invariant_valuation}).
    Using properties of good moduli spaces, this is enough to conclude (\S \ref{subsec:main_canonical}).
\end{enumerate}

We encounter some technical difficulties which are worth pointint out.

\begin{enumerate}[label=\textnormal{(\roman*)}]
    \item We need to keep track of the locus where $\mathscr{F}$ is not regular.
    We do so by adding a boundary divisor and using logarithmic geometry.
    \item We work with foliations which are not necessarily saturated in $\mathscr{T}_X$.
    We take ample space to thoroughly discuss foliations in this setting (\S \ref{sec:foliations}).
    \item The functorial semistable reduction only produces a Deligne--Mumford stack $Z^{\mathrm{ss}}$.
    For this reason, we must work with foliations on such spaces (\S \ref{subsec:foliations_dm_stacks}).
    \item If we merely required $g^{\mathrm{ss}}$ to be logarithmically smooth, we could take $\tau_Z$ to be a modification (\cite[Theorem 1.2.12]{abramovich2020}).
    However, we really do need semistability to ensure that $\mathscr{T}_{g^{\mathrm{ss}}}$ acts faithfully.
    \item The action of $G$ on $V$ need not have a fixed point.
    For this reason, we work with additive subgroups of $G$, which always have a fixed point (\S \ref{subsec:auxiliary}).
    \item There is no foundational material on formal stacks.
    As a result, our treatment may at times be overly laborious, especially in \S \ref{subsec:normal_representations}, \S \ref{subsec:relative_singularities} and \S \ref{subsec:formal_representability}.
\end{enumerate}

\subsection*{Notation and Assumptions}

We state the assumptions for each statement at the beginning of the corresponding subsection (\S \ref{sec:foliations} and \S \ref{sec:stacks}) or section (\S \ref{sec:resolution} and \S \ref{sec:main_theorem}).
We use standard letters $X, Y, \ldots$ for schemes and algebraic spaces, bold letters $\mathbf{X}, \mathbf{Y}, \ldots$ for Deligne--Mumford stacks, and calligraphic letters $\mathcal{X}, \mathcal{Y}, \ldots$ for algebraic stacks.
Every logarithmic scheme or Deligne--Mumford stack considered will always be fine and saturated.
Image and preimage will always be scheme-theoretic unless otherwise stated.
We use $X = Y$ to denote a natural isomorphism and $X \cong Y$ to denote a choice of isomorphism.
The relative sheaf of differentials and the relative tangent sheaf of a space $X$ over a field $\mathbbm{k}$ will simply be denoted by $\Omega_X^1$ and $\mathscr{T}_X$ respectively.
Given a coherent sheaf $\mathscr{F}$ on $X$ and a point $x \in X$, we let $\mathscr{F}_x$ and $\mathscr{F}|_x$ denote the localisation at $x$ and the fibre over $x$ respectively.
We use $\mathscr{F}^{\vee}$ and $\mathscr{F}^{[1]}$ for the dual and double dual respectively.
A point is not assumed to be closed. \par

We will use a few facts without further mention.
\begin{itemize}
    \item A locally Noetherian normal scheme is a disjoint union of integral schemes (\cite[\href{https://stacks.math.columbia.edu/tag/0357}{Lemma 0357}]{stacks-project}).
    \item The dual of a coherent sheaf on a normal scheme is reflexive (\cite[Corollary 1.2]{MR597077}), and a morphism of reflexive sheaves on a normal scheme is an isomorphism if and only if its localisation at every codimension one point is an isomorphism (\cite[Proposition 1.6]{MR597077}).
    \item The intersection of a discrete valuation ring with a subfield $K$ of its field of fractions is either $K$ or a discrete valuation ring.
\end{itemize}

\subsection*{Acknowledgements}

I am grateful to Caucher Birkar, Paolo Cascini, Michael Temkin and Jaroslaw Włodarczyk for discussions and clarifications about semistable reduction, and to Davide Gori, David Rydh and Mattia Talpo for answering questions related to algebraic stacks.
I would like to sincerely thank Fabio Bernasconi, Riccardo Carini, Paolo Cascini, Jorge Vitorio Pereira, Calum Spicer and Richard Thomas for helpful advice and encouragement.
I acknowledge funding from Tsinghua University where this research work was carried out.

    \section{Singularities of Foliations on Schemes}
\label{sec:foliations}

We start by defining foliations on schemes (\S \ref{subsec:foliations_schemes}) and logarithmic schemes (\S \ref{subsec:foliations_logarithmic}).
We then work towards defining singularities of foliations (\S \ref{subsec:discrepancies} and \S \ref{subsec:singularities}) by first studying foliations on discrete valuation rings (\S \ref{subsec:foliations_dvr}).
After describing the behaviour of singularities under smooth morphisms (\S \ref{subsec:ascent_descent}) and alterations (\S \ref{subsec:modifications_alterations}), we define the normal representations of stabiliser Lie algebras (\S \ref{subsec:normal_representations}) and compute them in the case of semistable morphisms (\S \ref{subsec:semistable_morphisms}).

\subsection{Foliations on Schemes}
\label{subsec:foliations_schemes}

We give basic definitions following \cite[\S \ref*{reg-sec:foliations}]{bongiorno3}. \par

In this subsection, $X$ is a locally Noetherian normal scheme over a field $\mathbbm{k}$ of characteristic zero.

\begin{definition}
    \label{def:foliation_scheme}
    A \emph{foliation} on $X$ is a coherent subsheaf $\mathscr{F}$ of the tangent sheaf $\mathscr{T}_X$ which is involutive, i.e. closed under the Lie bracket.
    A foliation is \emph{locally free} if $\mathscr{F}$ is a locally free sheaf.
    The \emph{rank} of $\mathscr{F}$ at a point $x \in X$ is the dimension of $\mathscr{F}|_x$.
    It is denoted by $\mathrm{rk}_x \, \mathscr{F}$.
    When $x$ is a generic point of $X$, we simply write $\mathrm{rk} \, \mathscr{F}$.
    A \emph{morphism of foliations} is a morphism of subsheaves of $\mathscr{T}_X$.
\end{definition}

Note that we do not require, a priori, the foliation to be saturated inside $\mathscr{T}_X$.

\begin{remark}
    \label{rem:foliation_is_morphism}
    Let $\mathscr{F}$ be a locally free foliation on $X$.
    By \cite[Proposition \ref*{reg-prop:foliation_groupoid}]{bongiorno3}, there exists a formally smooth groupoid $R$ on $X$ of formal finite presentation, whose associated foliation is $\mathscr{F}$.
    As a result, we may construct the associated infinitesimal stack
    \begin{align}
        \label{eq:formal_stack_foliation}
        X \rightarrow \mathcal{X} := [X/R],
    \end{align}
    and we may describe the locally free foliation $\mathscr{F}$ as the relative tangent sheaf of the formally smooth morphism in (\ref{eq:formal_stack_foliation}).
    In this context, an algebraically integrable foliation can be thought as the relative tangent sheaf of a morphism to an algebraisable formal stack $\mathcal{X}$.
    In this article, we want to resolve the singularities of this morphism in order to obtain a \emph{semistable} morphism.
    Similar techniques are used in \cite{abramovich2025}.
\end{remark}

\begin{definition}
    \label{def:invariant_subscheme}
    Let $\mathscr{F}$ be a locally free foliation on $X$ and let $V \subseteq X$ be a locally closed subscheme with ideal sheaf $\mathscr{I}_V$.
    Then $V$ is \emph{invariant} with respect to $\mathscr{F}$ if for all local derivations $\partial$ of $\mathscr{F}$, $\partial(\mathscr{I}_V) \subseteq \mathscr{I}_V$.
\end{definition}

We invite the reader to consult \cite[Lemma \ref*{reg-lem:invariance_equivalence}]{bongiorno3} for a stack-theoretic characterisation of invariance: a subscheme is invariant if and only if it descends to a closed substack on the associated formal stack.

\subsection{Foliations on Logarithmic Schemes}
\label{subsec:foliations_logarithmic}

We generalise the previous subsection to the setting of logarithmic schemes $(X, \Delta_X)$.
Setting $\Delta_X = \emptyset$ recovers the standard case. \par

In this subsection, $X$ and $Y$ are locally Noetherian normal schemes over a field $\mathbbm{k}$ of characteristic zero, and $\Delta_X \subseteq X$ and $\Delta_Y \subseteq Y$ are reduced Cartier divisors on $X$ and $Y$ respectively.
We let $f : (X, \Delta_X) \rightarrow (Y, \Delta_Y)$ be a logarithmic morphism over $\mathbbm{k}$.
We assume $x \in X$ is a point and set $y := f(x) \in Y$.

\begin{definition}
    \label{def:log_tangent_sheaf}
    The \emph{relative logarithmic tangent sheaf} of $f$ is defined as the dual of the sheaf of relative logarithmic differentials
    \begin{align}
        \label{eq:log_tangent_sheaf}
        \mathscr{T}_{X/Y}(-\mathrm{log}\,\Delta_X/\Delta_Y) = \left( \Omega_{X/Y}^1(\mathrm{log}\,\Delta_X/\Delta_Y) \right)^{\vee}.
    \end{align}
    Since it is closed under the Lie bracket, it is a foliation on $X$.
    For ease of notation, these objects will be denoted by $\mathscr{T}_f$ and $\Omega_f^1$ respectively.
\end{definition}

\begin{remark}
    \label{rem:log_tangent_sheaf}
    It is a standard fact that a local derivation $\partial$ of $\mathscr{T}_X$ is in $\mathscr{T}_f$ if and only if it is a derivation on $X$ relative to $Y$ under which $\Delta_X$ is invariant.
\end{remark}

\begin{definition}
    \label{def:log_foliation}
    A \emph{foliation} on $(X ,\Delta_X)$ is a coherent subsheaf $\mathscr{F}$ of the logarithmic tangent sheaf $\mathscr{T}_{X}(-\mathrm{log}\,\Delta_X)$ which is involutive.
    The \emph{log-normal sheaf} of $\mathscr{F}$ is
    \begin{align}
        \label{eq:normal_sheaf_foliation}
        \mathscr{N}_{\mathscr{F}} := \mathrm{coker} \left( \mathscr{F} \rightarrow \mathscr{T}_X \left( - \mathrm{log}\,\Delta_X \right) \right).
    \end{align}
    The foliation $\mathscr{F}$ is \emph{log-saturated} at $x$ if its log-normal sheaf is torsion-free at $x$.
    When $x \notin \Delta_X$, we simply say $\mathscr{F}$ is \emph{saturated} at $x$.
\end{definition}

\begin{definition}
    \label{def:log_regular}
    Let $\mathscr{F}$ be a foliation on $(X, \Delta_X)$.
    Then $\mathscr{F}$ is \emph{log-regular} at $x$ if $\mathscr{F}$ is locally free at $x$ and
    \begin{align}
        \label{eq:regular_locus_morphism}
        \Omega_X^1(\mathrm{log}\,\Delta_X) \rightarrow \mathscr{F}^{\vee}
    \end{align}
    is surjective at $x$.
    The \emph{log-regular locus} of $\mathscr{F}$ is the open subset of log-regular points of $X$.
    When $x \notin \Delta_X$, we simply say $x$ is a \emph{regular} point.
\end{definition}

\begin{lemma}
    \label{lem:log_regular_log_saturated}
    Let $\mathscr{F}$ be a foliation on $(X, \Delta_X)$.
    If $\mathscr{F}$ is log-regular at $x$, it is log-saturated at $x$.
\end{lemma}

\begin{proof}
    We work locally around $x$.
    If $\mathscr{F}$ is log-regular, we may write a short exact sequence
    \begin{align}
        \label{eq:regular_kernel}
        0 \rightarrow \mathscr{K} \rightarrow \Omega_X^1(\mathrm{log}\,\Delta_X) \rightarrow \mathscr{F}^{\vee} \rightarrow 0.
    \end{align}
    Taking the dual of (\ref{eq:regular_kernel}) yields an exact sequence
    \begin{align}
        \label{eq:regular_kernel_dual}
        0 \rightarrow \mathscr{F} \rightarrow \mathscr{T}_X(-\mathrm{log}\,\Delta_X) \rightarrow \mathscr{K}^{\vee},
    \end{align}
    from which it follows that $\mathscr{N}_{\mathscr{F}}$ is a subsheaf of a torsion free sheaf, hence it is torsion free.
\end{proof}

\begin{lemma}
    \label{lem:regular_foliations_isomorphic}
    Let $\mathscr{F} \subseteq \mathscr{F}^{\prime}$ be two foliations on $(X, \Delta_X)$.
    Assume that $\mathscr{F}$ and $\mathscr{F}^{\prime}$ have the same rank at the unique generic point generalising $x$.
    If $\mathscr{F}$ is log-regular at $x$, then $\mathscr{F} = \mathscr{F}^{\prime}$ locally around $x$.
\end{lemma}

\begin{proof}
    Recall that, on a normal scheme, there exists a unique generic point generalising $x$.
    We work locally around $x$.
    Consider the induced morphism of log-normal sheaves $\mathscr{N}_{\mathscr{F}} \rightarrow \mathscr{N}_{\mathscr{F}^{\prime}}$.
    By assumption, this is a surjective morphism of sheaves of the same generic rank.
    By Lemma \ref{lem:log_regular_log_saturated}, $\mathscr{N}_{\mathscr{F}}$ is torsion-free, therefore $\mathscr{N}_{\mathscr{F}} = \mathscr{N}_{\mathscr{F}^{\prime}}$ and $\mathscr{F} = \mathscr{F}^{\prime}$.
\end{proof}

\begin{definition}
	\label{def:gorenstein}
	Let $\mathscr{F}$ be a foliation on $(X, \Delta_X)$.
	The \emph{canonical sheaf} of $\mathscr{F}$ is
	\begin{align}
		\omega_{\mathscr{F}} &= \left( \Lambda^{\mathrm{top}} \mathscr{F} \right)^{\vee}, \text{ and} \\
        \omega_{\mathscr{F}}^{[m]} &= \left( \omega_{\mathscr{F}}^{\otimes m} \right)^{[1]}
	\end{align}
    is its $m\textsuperscript{th}$ reflexive power for any $m \in \mathbb{N}$.
    Here, $\mathrm{top}$ is the rank of $\mathscr{F}$ at the unique generic point generalising $x$.
	The foliation is \emph{$\mathbb{Q}$-Gorenstein} at $x$ if $\omega_{\mathscr{F}}^{[m]}$ is invertible at $x$ for some $m \in \mathbb{N}$.
    The least such $m$ for which this holds is the \emph{Gorenstein index} of $\mathscr{F}$.
    When $m = 1$, we say that $\mathscr{F}$ is \emph{Gorenstein}.
\end{definition}

Note that a locally free foliation is Gorenstein.

\begin{definition}
    \label{def:pullback_foliation}
    Suppose that $f$ is flat, then we define the \emph{logarithmic pullback foliation} $\mathscr{F}$ on $(X, \Delta_X)$ as the limit of the diagram
    \begin{equation}
        \label{diag:pullback_foliation}
        \begin{tikzcd}
            \mathscr{F} \arrow[r, dashed] \arrow[d, dashed] & \mathscr{T}_X(-\mathrm{log}\,\Delta_X) \arrow[d] \\
            f^* \mathscr{G} \arrow[r] & f^* \mathscr{T}_Y(-\mathrm{log}\, \Delta_Y).
        \end{tikzcd}
    \end{equation}
\end{definition}

\begin{remark}
    \label{rem:pullback_foliation}
    We make a few remarks about Definition \ref{def:pullback_foliation},
    \begin{itemize}
        \item Flatness of $f$ ensures existence of the morphism between logarithmic tangent sheaves.
        \item Since $f^* \mathscr{G} \subseteq f^* \mathscr{T}_Y(-\mathrm{log}\, \Delta_Y)$ is injective, so is $\mathscr{F} \subseteq \mathscr{T}_X(-\mathrm{log}\, \Delta_X)$.
        \item Since $\mathscr{F}$ is a limit of coherent sheaves closed under the Lie bracket, we see that $\mathscr{F}$ inherits the same properties, thus it is a foliation on $(X, \Delta_X)$.
        \item If $f$ is formally étale, e.g. a localisation or an étale morphism, $\mathscr{F} = f^* \mathscr{G}$.
        \item If $f$ is a locally closed immersion of a $\mathscr{G}$-invariant logarithmic subscheme, $\mathscr{F} = f^* \mathscr{G}$.
        \item When $\Delta_X$ and $\Delta_Y$ are the empty log structures, this is the usual pullback foliation, however we do not saturate $\mathscr{F}$ inside $\mathscr{T}_X$.
    \end{itemize}
\end{remark}

\begin{lemma}
    \label{lem:log_pullback}
    Let $\mathscr{G}$ be a foliation on $Y$ and suppose that $f$ is flat.
    Let $\mathscr{F}$ denote the logarithmic pullback foliation of $\mathscr{G}$ on $X$.
    Then,
    \begin{enumerate}
        \item there exists an exact sequence
        \begin{align}
            \label{eq:relative_foliations_ses}
            0 \rightarrow \mathscr{T}_f \rightarrow \mathscr{F} \rightarrow f^* \mathscr{G}, \text{ and}
        \end{align}
        \item if $\mathscr{G}$ is log-saturated at $y$, so is $\mathscr{F}$ at $x$.
    \end{enumerate}
    If furthermore $f$ is logarithmically smooth at $x$, then
    \begin{enumerate}[resume]
        \item the exact sequence in (\ref{eq:relative_foliations_ses}) is short exact at $x$, and
        \item $\omega_{\mathscr{F}} = \omega_f \otimes f^* \omega_{\mathscr{G}}$ in a neighbourhood of $x$.
    \end{enumerate}
\end{lemma}

We point out that the exact sequence in (\ref{eq:relative_foliations_ses}) is simply the dual of the exact sequence of sheaves of differentials associated to the morphisms $X \rightarrow Y \rightarrow [Y / \mathscr{G}]$.

\begin{proof}
    [Proof of Lemma \ref{lem:log_pullback}]
    We work locally around $x \in X$ and $y \in Y$.
    \begin{enumerate}[itemsep=1em]
        \item Flatness of $f$ yields existence of a morphism of short exact sequences
        \begin{equation}
            \label{diag:pullback_foliation_ses}
            \begin{tikzcd}
                0 \arrow[r] & \mathscr{F} \arrow[r] \arrow[d] & \mathscr{T}_X(-\mathrm{log}\,\Delta_X) \arrow[r] \arrow[d] & \mathscr{N}_{\mathscr{F}} \arrow[r] \arrow[d] & 0 \\
                0 \arrow[r] & f^*\mathscr{G} \arrow[r] & f^*\mathscr{T}_Y(-\mathrm{log}\,\Delta_Y) \arrow[r] & f^*\mathscr{N}_{\mathscr{G}} \arrow[r] & 0.
            \end{tikzcd}
        \end{equation}
        Now
        \begin{align}
            \label{eq:kernel_log_tangent}
            \mathscr{T}_f = \mathrm{ker} \, \Bigl( \mathscr{T}_X(-\mathrm{log}\,\Delta_X) \rightarrow f^*\mathscr{T}_Y(-\mathrm{log}\,\Delta_Y) \Bigr)
        \end{align}
        and we need to show it factors through $\mathscr{F} \subseteq \mathscr{T}_X(-\mathrm{log}\,\Delta_X)$.
        To this end, it suffices to show that $\mathscr{N}_{\mathscr{F}} \rightarrow f^*\mathscr{N}_{\mathscr{G}}$ is injective.
        Since $\mathscr{F}$ is defined as the limit of the left-most square, injectivity follows from a simple diagram chasing argument.

        \item By assumption, $\mathscr{N}_{\mathscr{G}}$ is torsion-free.
        Since $f$ is flat, $f^*\mathscr{N}_{\mathscr{G}}$ is torsion-free.
        But now, by part (1), $\mathscr{N}_{\mathscr{F}}$ is a subsheaf of $f^*\mathscr{N}_{\mathscr{G}}$, thus it is also torsion-free.

        \item Logarithmic smoothness of $f$ implies that the morphism
        \begin{align}
            \label{eq:log_tangent_map_surjective}
            \mathscr{T}_X(-\mathrm{log}\,\Delta_X) \rightarrow f^*\mathscr{T}_Y(-\mathrm{log}\,\Delta_Y)
        \end{align}
        is surjective.
        Using once more injectivity of $\mathscr{N}_{\mathscr{F}} \rightarrow f^*\mathscr{N}_{\mathscr{G}}$, we have that $\mathscr{F} \rightarrow f^*\mathscr{G}$ is surjective.

        \item Taking exterior powers of the short exact sequence (part (3)) in (\ref{eq:relative_foliations_ses}) and dualising yields the result, after using the fact that $\mathscr{T}_f$ is locally free, and $X$ and $Y$ are normal.
    \end{enumerate}
\end{proof}

\subsection{Foliations on Discrete Valuation Rings}
\label{subsec:foliations_dvr}

Discrepancies of foliations will be defined on divisorial valuations.
We start preparing by studying foliations on discrete valuation rings. \par

In this subsection, $D$ and $E$ are discrete valuation rings or fields over a field $\mathbbm{k}$ of characteristic zero.
We let $d$ and $e$ denote the unique closed points of $D$ and $E$ respectively.
We also let $K$ and $L$ denote the fields of fractions of $D$ and $E$ respectively.
We let $\Delta_D$ and $\Delta_E$ be either logarithmic structure on $\spec D$ and $\spec E$ respectively: the closed point or the empty set.
We fix a local inclusion of local rings $D \subseteq E$ and let $g : (\spec E, \Delta_E) \rightarrow (\spec D, \Delta_D)$ denote the induced morphism of spectra.
We assume $g$ is a logarithmic morphism of logarithmic schemes.

\begin{lemma}
    \label{lem:dvr_log_saturated}
    Let $\mathscr{F}$ be a foliation on $(E, \Delta_E)$.
    Then, $\mathscr{F}$ is log-regular if and only if it is log-saturated.
\end{lemma}

\begin{proof}
    We have already seen the sufficient implication in Lemma \ref{lem:log_regular_log_saturated}.
    For the necessary implication, consider the exact sequence
    \begin{align}
        \label{eq:normal_ses}
        0 \rightarrow \mathscr{F} \rightarrow \mathscr{T}_E(-\mathrm{log} \, \Delta_E) \rightarrow \mathscr{N}_{\mathscr{F}} \rightarrow 0.
    \end{align}
    Note that both $\mathscr{F}$ and the log-normal sheaf of $\mathscr{F}$ are torsion-free, thus locally free over $E$, hence dualising (\ref{eq:normal_ses}) yields log-regularity.
\end{proof}

\begin{lemma}
    \label{lem:pullback_foliation_dvr}
    Let $\mathscr{G}$ be a log-regular foliation on $(\spec D, \Delta_D)$.
    Let $\mathscr{F}$ and $\mathscr{F}_L$ be the logarithmic pullback foliations on $(\spec E, \Delta_E)$ and $\spec L$ respectively of $\mathscr{G}$ and $\mathscr{G}|_K$ respectively.
    Then $\mathscr{F}$ is log-regular and $\mathscr{F}|_L = \mathscr{F}_L$.
\end{lemma}

\begin{proof}
    Firstly note that $g$ is automatically flat, hence we may define logarithmic pullback foliations.
    The fact that $\mathscr{F}$ is log-regular follows by combining part (2) of Lemma \ref{lem:log_pullback} with Lemma \ref{lem:dvr_log_saturated}.
    Since $\mathscr{F}$ is the logarithmic pullback and limits commute with localisations, $\mathscr{F}|_L = \mathscr{F}_L$.
\end{proof}

\begin{construction}
    \label{cons:regular_foliation_dvr}
    Let $\mathscr{F}_L$ be a foliation on $\spec L$ and let $\iota : \spec L \rightarrow \spec E$ denote the inclusion of the generic point of $E$.
    We construct two unique natural foliations on $\spec E$:
    \begin{enumerate}
        \item[($\emptyset$)] $\mathscr{F}_E$, the unique regular foliation on $\spec E$ such that $\iota^* \mathscr{F}_E = \mathscr{F}_L$, and
        \item[(e)] $\mathscr{F}_E^{e}$, the unique log-regular foliation on $(\spec E, e)$ such that $\iota^* \mathscr{F}_E^{e} = \mathscr{F}_L$.
    \end{enumerate}
    Since $\mathscr{T}_E(-\mathrm{log} \, \Delta_E)$ is torsion free, the morphism
    \begin{align}
        \label{eq:tangent_sheaf_torsion_free}
        \mathscr{T}_E(-\mathrm{log} \, \Delta_E) \rightarrow \iota_* \iota^* \mathscr{T}_E(-\mathrm{log} \, \Delta_E) = \iota_* \mathscr{T}_L
    \end{align}
    is injective.
    Since $\iota$ is affine, the morphism $\iota_* \mathscr{F}_L \rightarrow \iota_* \mathscr{T}_L$ is injective.
    Let
    \begin{align}
        \label{eq:log_regular_foliation_def}
        \mathscr{F}_E^{\Delta} := \iota_* \mathscr{F}_L \cap \mathscr{T}_E(-\mathrm{log} \, \Delta_E) \subseteq \iota_* \mathscr{T}_L,
    \end{align}
    and define $\mathscr{F}_E$ and $\mathscr{F}_E^{e}$ to be $\mathscr{F}_E^{\Delta}$ when $\Delta_E$ is the empty set or $e$ respectively.
    We need to prove five facts.
    \begin{enumerate}[itemsep=1em]
        \item $\mathscr{F}_E^{\Delta}$ is a coherent subsheaf of $\mathscr{T}_E(-\mathrm{log} \, \Delta_E)$.
        This is true by construction and the fact that $E$ is Noetherian.
        
        \item $\mathscr{F}_E^{\Delta}$ is involutive.
        This follows from the fact that $\mathscr{F}_E^{\Delta}$ is the intersection of two involutive subsheaves of $\iota_* \mathscr{T}_L$.
        
        \item $\iota^* \mathscr{F}_E^{\Delta} = \mathscr{F}_L$.
        By construction, there is a natural morphism $\iota^* \mathscr{F}_E^{\Delta} \rightarrow \mathscr{F}_L$.
        This is injective.
        Indeed, since $\iota$ is flat, the morphism $\iota^* \mathscr{F}_E^{\Delta} \rightarrow \mathscr{T}_L$ is injective.
        To see surjectivity, let $\partial \in \mathscr{F}_L$.
        Because $\iota^* \mathscr{T}_E(-\mathrm{log} \, \Delta_E) = \mathscr{T}_L$, after clearing all denominators, we may find $a \in E$ such that $a \cdot \partial \in \mathscr{T}_E(-\mathrm{log} \, \Delta_E)$.
        But then, $a \cdot \partial \in \mathscr{F}_E^{\Delta}$ and $(a \cdot \partial) \otimes 1/a \rightarrow \partial \in \mathscr{F}_L$.
        
        \item $\mathscr{F}_E^{\Delta}$ is log-regular.
        Using (3) and the fact that $\iota$ is flat and affine, we obtain a morphism of short exact sequences
        \begin{equation}
            \label{diag:ses_generic_point}
            \begin{tikzcd}
                0 \arrow[r] & \mathscr{F}_E^{\Delta} \arrow[r] \arrow[d] & \mathscr{T}_E \left( - \mathrm{log}\,\Delta_E \right) \arrow[r] \arrow[d] & \mathscr{N}_{\mathscr{F}_E^{\Delta}} \arrow[r] \arrow[d] & 0 \\
                0 \arrow[r] & \iota_*\mathscr{F}_L \arrow[r] & \iota_* \mathscr{T}_L \arrow[r] & \iota_* \iota^* \mathscr{N}_{\mathscr{F}_E^{\Delta}} \arrow[r] & 0.
            \end{tikzcd}
        \end{equation}
        Because $\mathscr{F}_E^{\Delta}$ is defined as a limit, a simple diagram chasing argument shows that $\mathscr{N}_{\mathscr{F}_E^{\Delta}} \rightarrow \iota_* \iota^* \mathscr{N}_{\mathscr{F}_E^{\Delta}}$ is injective.
        This is the definition of torsion-free and we conclude by the necessary implication of Lemma \ref{lem:dvr_log_saturated}.
        
        \item $\mathscr{F}_E^{\Delta}$ is the unique foliation on $(E, \Delta_E)$ satisfying (3) and (4).
        Let $\mathscr{H}$ be another such foliation. Since $\mathscr{F}_E^{\Delta}$ is defined as a limit, we obtain a unique morphism of foliations $\mathscr{H} \rightarrow \mathscr{F}_E^{\Delta}$.
        But then the induced morphism of log-normal sheaves $\mathscr{N}_{\mathscr{H}} \rightarrow \mathscr{N}_{\mathscr{F}_E^{\Delta}}$ is a surjective morphism of locally free sheaves of the same rank, thus it is an isomorphism.
    \end{enumerate}
\end{construction}

\begin{lemma}
    \label{lem:dvr_log_smooth}
    Suppose that $E$ is a discrete valuation ring and that $K \subseteq L$ is a finite field extension.
    Then $D$ is a discrete valuation ring and $g$ induces an isomorphism of logarithmic tangent sheaves
    \begin{align}
        \label{eq:log_tangent_sheaf_dvr}
        g^* \mathscr{T}_D(-\mathrm{log} \, d) = \mathscr{T}_E(-\mathrm{log} \, e).
    \end{align}
\end{lemma}

\begin{proof}
    We first show that $D$ must be a discrete valuation ring.
    Note that $E \cap K$ is either a discrete valuation ring or is equal to $K$.
    If $E \cap K$ were equal to $K$, there would morphisms $K \rightarrow E \rightarrow L$.
    Since $L$ is finite over $K$, $E$ would necessarily be a field, a contradiction. \par
        
    We obtain an associated exact sequence
    \begin{align}
        \label{eq:dvr_log_differential_ses}
        g^* \Omega_D^1(\mathrm{log}\, d) \rightarrow \Omega_E^1(\mathrm{log}\, e) \rightarrow \Omega_g^1 \rightarrow 0
    \end{align}
    Since $\Omega_D^1(\mathrm{log}\, d)$ is locally free, (\ref{eq:dvr_log_differential_ses}) is in fact short exact.
    Furthermore, since $g$ is flat, dual commutes with pullback and it suffices to show that $\Omega_g^1 = 0$.
    We check this with a computation, upon performing some initial reductions. \par
    
    Firstly, we may assume that the discrete valuations rings are complete.
    Indeed $\Omega_g^1 = 0$ if and only if its restriction to the closed point $e$ is zero.
    Let $k_D$ and $k_E$ be the residue fields of the closed points $d$ and $e$ respectively.
    By Cohen's structure theorem (\cite[\href{https://stacks.math.columbia.edu/tag/0C0S}{Lemma 0C0S}]{stacks-project}), we may assume that $D \cong k_D \llbracket s \rrbracket$ and $E \cong k_E \llbracket t \rrbracket$ for local parameters $s \in D$ and $t \in E$, and coefficient fields $k_D \subseteq D$ and $k_E \subseteq E$.
    Furthermore, since $K \subseteq L$ is a finite field extension, the Krull--Akizuki theorem implies that $k_D \rightarrow k_E$ is a finite morphism, hence, in characteristic zero, étale.
    It follows that the base change $k_D \llbracket s \rrbracket \rightarrow k_E \llbracket s \rrbracket$ is étale, thus we may replace $D = k_D \llbracket s \rrbracket$ with $k_E \llbracket s \rrbracket$, i.e. we may assume that $k_D = k_E$. \par

    Now, the morphism $D \rightarrow E$ is determined by the image of $s$, which we may write as $v \, t^n$ for a unit $v \in k_E^{\times}$ and a positive integer $n$.
    Up to replacing $k_D = k_E$ with the splitting field of $v^n - 1$, we may assume that the $n\textsuperscript{th}$ root of unity of $v$ lies in $k_D$.
    Indeed, in characteristic zero, such splitting field is étale over $k_D$.
    As a result, up to dividing $s$ by the root of unity of $v$, we may assume that $s \rightarrow t^n$. \par

    Under these reductions, $\Omega_g^1$ is the cokernel of
    \begin{align}
        \label{eq:dvr_log_differential_cokernel}
        k_E\llbracket t \rrbracket \left\langle \frac{ds}{s} \right\rangle &\rightarrow k_E\llbracket t \rrbracket \left\langle \frac{dt}{t} \right\rangle \\
        \frac{ds}{s} &\rightarrow \frac{nt^{n-1}dt}{t^n} = n \frac{dt}{t}.\nonumber
    \end{align}
    Since $n$ is invertible in $k_E$, $\Omega_g^1 = 0$.
\end{proof}

\subsection{Discrepancies of Foliations}
\label{subsec:discrepancies}

We follow the main ideas of \cite[\S I.ii]{MR3128985} in order to define discrepancies of foliations. \par

In this subsection, $X$ is a locally Noetherian normal scheme over a field $\mathbbm{k}$ of characteristic zero, $\Delta_X$ is a reduced Cartier divisor on $X$, and $x \in X$ is a point.

\begin{construction}
    \label{cons:discrepancy}
    Let $\mathscr{F}$ be a Gorenstein foliation on $(X, \Delta_X)$.
    We construct \emph{discrepancies} of $\mathscr{F}$ associated to divisorial valuations over $X$.
    Let $\pi : \spec E \rightarrow X$ be a divisorial valuation over $X$ and let $e$ be its closed point.
    Let $\Delta_E$ be a logarithmic structure on $\spec E$ such that $\pi$ is a logarithmic morphism.
    We have that $\Delta_E$ is either the empty set or the closed point $e$.
    By definition, $\pi$ induces an isomorphism between the generic point of $\spec E$ and the generic point of $X$.
    Let $L$ denote the field of fractions, we get logarithmic morphisms
    \begin{align}
    	\label{eq:generic_point_dvr}
		\spec L \xrightarrow{\iota} \spec (E, \Delta_E) \xrightarrow{\pi} (X, \Delta_X).
	\end{align}
	Let $\mathscr{F}_{L} = \iota^* \pi^* \mathscr{F}$ be the foliation induced on the generic point $\spec L$.
    We obtain a unique log-regular foliation $\mathscr{F}_E^{\Delta}$ on $\spec (E, \Delta_E)$ from Construction \ref{cons:regular_foliation_dvr}.
    We want to measure how distant $\mathscr{F}$ is from the log-regular foliation $\mathscr{F}_E^{\Delta}$.
    To this end, we compare the canonical sheaves of $\pi^* \mathscr{F}$ and $\mathscr{F}_E^{\Delta}$.
	Since $\mathscr{F}_E^{\Delta}$ is locally free and localisation commutes with exterior powers and duals, we have that $\omega_{\mathscr{F}_{L}}^{\vee} = \iota^*\omega_{\mathscr{F}_E^{\Delta}}^{\vee}$.
    Furthermore, since $\mathscr{F}$ is Gorenstein, we have that $\omega_{\mathscr{F}_{L}} = \iota^* \pi^* \omega_{\mathscr{F}}$.
    This yields morphisms 
	\begin{align}
		\label{eq:canonical_generic_inclusions}
        \omega_{\mathscr{F}_E^{\Delta}}^{\vee} &\rightarrow \iota_*\omega_{\mathscr{F}_{L}}^{\vee}, \text{ and} \\
		\pi^* \omega_{\mathscr{F}} &\rightarrow \iota_*\omega_{\mathscr{F}_{L}}.
	\end{align}
    In turn, we get a natural morphism
    \begin{align}
    	\label{eq:discrepancy_morphism}
        \varphi_E^{\Delta} : 
		\omega_{\mathscr{F}_E^{\Delta}}^{\vee} \otimes \pi^* \omega_{\mathscr{F}} \rightarrow
        \iota_*\omega_{\mathscr{F}_{L}}^{\vee} \otimes \iota_*\omega_{\mathscr{F}_{L}} \rightarrow
        \iota_* \left( \omega_{\mathscr{F}_{L}}^{\vee} \otimes \omega_{\mathscr{F}_{L}} \right) \rightarrow
        \iota_* \mathscr{O}_{\spec L} = L.
	\end{align}
    Let $s$ be a local generator of the invertible sheaf $\omega_{\mathscr{F}_E^{\Delta}}^{\vee} \otimes \pi^* \omega_{\mathscr{F}}$.
    Define the \emph{discrepancy} and \emph{log-discrepancy} of $(X, \mathscr{F})$ with respect to $E$ to be respectively
    \begin{align}
        \label{eq:discrepancy}
    	a_E(X, \mathscr{F}) &:= \nu_E(\varphi_E^{\Delta = \emptyset}(s)) \in \mathbb{Z}, \text{ and} \\
        \label{eq:log_discrepancy}
        a_E^{e}(X, \mathscr{F}) &:= \nu_E(\varphi_E^{\Delta = e}(s)) \in \mathbb{Z},
    \end{align}
    where $\nu_E$ is the discrete valuation of $E$.
    Clearly, the discrepancies are independent of the choice of generator.
    Note that, in order to define the non-log-discrepancy of $E$, we must have that $\pi^{-1}(\Delta_X) = \emptyset$.
\end{construction}

\begin{remark}
    \label{rem:consistency_epsilon}
    Let $\mathscr{F}$ be a Gorenstein saturated foliation on $X$.
    We note that this definition is consistent with the usual definition of discrepancy.
    It is well-known that, given any divisorial valuation $E$ over $X$ with closed point $e$, we may find a proper birational morphism $\pi : \tilde{X} \rightarrow X$ such that $\tilde{X}$ is normal, $e$ is a regular codimension one point of $\tilde{X}$ and $\pi$ is a sequence of blowing-ups in smooth centres (\cite[Lemma 2.45]{MR1658959}).
    We now consider the pullback foliation $\tilde{\mathscr{F}}$ on $\tilde{X}$.
    Note that, as $\pi$ is not flat, we cannot define the pullback foliation in our setup.
    However, one can show that there exists a unique saturated foliation $\tilde{\mathscr{F}}$ on $\tilde{X}$ which is isomorphic to $\mathscr{F}$ over the generic point of $X$.
    Since $e$ is a regular codimension one point, we may define the discrepancy of $\mathscr{F}$ with respect to $E$ as
    \begin{align}
        \label{eq:mmp_discrepancy}
        a_E(X, \mathscr{F}) = -\mathrm{ord}_E \, \left(\omega_{\tilde{\mathscr{F}}} \otimes \pi^* \omega_{\mathscr{F}}^{\vee} \right) \in \mathbb{Z}.
    \end{align}
    It is straightforward to see that the localisation of $\tilde{\mathscr{F}}$ at $e$ is the foliation $\mathscr{F}_E$ of Construction \ref{cons:regular_foliation_dvr}, thus the definitions agree.
    We may also define the log-discrepancy of $\mathscr{F}$ as
    \begin{align}
        \label{eq:epsilon}
        a_E^{e}(X, \mathscr{F}) &= a_E(X, \mathscr{F}) + \epsilon(E),
        \text{ where}  \\
        \epsilon(E) &=
        \begin{cases} 
            0, & \text{ if $e$ is $\mathscr{F}_E$-invariant,} \\
            1, & \text{ otherwise.}
        \end{cases}
        \nonumber
    \end{align}
    By definition, $\epsilon(E) = 0$ if and only if the inclusion $\mathscr{F}_E^e \subseteq \mathscr{F}_E$ is an equality.
    In fact, by standard logarithmic geometry, it follows that
    \begin{align}
        \label{eq:mmp_log_discrepancy}
        \epsilon(E) = \mathrm{ord}_E \, \left(\omega_{\mathscr{F}_E^e} \otimes \omega_{\mathscr{F}_E}^{\vee} \right).
    \end{align}
    As a result, the definitions of log-discrepancy also agree.
\end{remark}

\begin{remark}
    \label{rem:foliated_triples}
    Typically, in the setting of the Minimal Model Program of foliations, the objects of study are foliated triples $(X, \mathscr{F}, \Delta_X)$.
    This means that $\mathscr{F}$ is a saturated foliation on $X$ and $\Delta_X$ is a $\mathbb{Q}$-divisor with coefficients in $[0, 1]$ such that $K_{\mathscr{F}} + \Delta_X$ is $\mathbb{Q}$-Cartier.
    Our definitions adapt to this case.
    Suppose for simplicity that $\Delta_X$ and $K_{\mathscr{F}} + \Delta_X$ are reduced Cartier divisors (see Remark \ref{rem:foliated_triples_q} for the more general case).
    We consider the (unsaturated) foliation
    \begin{align}
        \label{eq:log_foliation}
        \mathscr{F}_X^{\Delta} = \mathscr{F} \cap \mathscr{T}_X(-\mathrm{log}\,\Delta_X) \subseteq \mathscr{T}_X.
    \end{align}
    Note that $\mathscr{F}_X^{\Delta} = \mathscr{F}$ in a neighbourhood of all prime divisors in the support of $\Delta_X$ which are $\mathscr{F}$-invariant.
    Otherwise, such prime divisor is $\mathscr{F}_X^{\Delta}$-invariant but not $\mathscr{F}$-invariant.
    By standard logarithmic geometry, it follows that
    \begin{align}
        \label{eq:log_canonical_bundles}
        \omega_{\mathscr{F}_X^{\Delta}} &= \omega_{\mathscr{F}} \otimes \mathscr{O}_X(\Delta_X), \text{ i.e.} \\
        \label{eq:log_canonical_divisor}
        K_{\mathscr{F}_X^{\Delta}} &= K_{\mathscr{F}} + \Delta_X,
    \end{align}
    and, in our terminology, $\mathscr{F}_X^{\Delta}$ is a foliation on $(X, \Delta_X)$.
\end{remark}

\subsection{Singularities of Foliations}
\label{subsec:singularities}

We now use discrepancies to define canonical and log-canonical singularities of foliations.
We also show how these singularities relate with regularity and saturation. \par

In this subsection, $X$ is a locally Noetherian normal scheme over a field $\mathbbm{k}$ of characteristic zero, $\Delta_X$ is a reduced Cartier divisor on $X$, and $x \in X$ is a point.

\begin{definition}
    \label{def:singularities}
    Let $\mathscr{F}$ be a foliation on $(X, \Delta_X)$.
    Then $\mathscr{F}$ is \emph{(log-)canonical} at $x$ if $\mathscr{F}$ is Gorenstein at $x$ and any divisorial valuation over $X$ with centre containing $x$ has non-negative (log-)discrepancy, i.e. the image
    \begin{align}
        \label{eq:discrepancy_image}
        \mathrm{im}\left( \varphi_E^{\Delta} \right) \subseteq L
    \end{align}
    factors through $E \subseteq L$.
\end{definition}

Note that for a singularity to be canonical at $x$, we must have that $x \notin \Delta_X$.

\begin{lemma}
    \label{lem:regular_implies_canonical}
    Let $\mathscr{F}$ be a foliation on $(X, \Delta_X)$.
    If $\mathscr{F}$ is (log-)regular at $x$, $\mathscr{F}$ is (log-)canonical at $x$.
\end{lemma}

\begin{proof}
    We work locally around $x$.
    By definition of regularity, $\mathscr{F}$ is Gorenstein.
    Let $\pi : \spec E \rightarrow X$ be a divisorial valuation and set $\Delta_E = \emptyset$ if $x \notin \Delta_X$ and $\Delta_E = e$ if $x \in \Delta_X$, where $e$ is the unique closed point of $E$.
    With this convention, $\pi : (\spec E, \Delta_E) \rightarrow (X, \Delta_X)$ is a logarithmic morphism.
    By assumption, we have a surjection $\Omega_X^1(\mathrm{log}\,\Delta_X) \rightarrow \mathscr{F}^{\vee}$.
    Let $\mathscr{K}$ be its kernel.
    We get exact sequences
    \begin{equation}
        \label{diag:regular_implies_canonical}
        \begin{tikzcd}
            & \pi^*\mathscr{K} \arrow[r] & \pi^*\Omega_X^1(\mathrm{log} \, \Delta_X) \arrow[r] \arrow[d] & \pi^*\mathscr{F}^{\vee} \arrow[r] & 0 \\
            0 \arrow[r] & \mathscr{N}^{\vee}_{\mathscr{F}_E^{\Delta}} \arrow[r] & \Omega_E^1(\mathrm{log} \, \Delta_E) \arrow[r] & {\mathscr{F}_E^{\Delta}}^{\vee} \arrow[r] & 0.
        \end{tikzcd}
    \end{equation}
    To show that the discrepancy is zero, we show that the morphism between logarithmic sheaves of differentials descends to a morphism
    \begin{align}
        \label{eq:regular_foliations_discrepancy}
        \pi^*\mathscr{F}^{\vee} \rightarrow {\mathscr{F}_E^{\Delta}}^{\vee}.
    \end{align}
    Indeed, taking top exterior powers of (\ref{eq:regular_foliations_discrepancy}) and using local freeness shows that
    \begin{align}
        \label{eq:discrepancy_morphism_canonical}
        \pi^*\omega_{\mathscr{F}} &\rightarrow \omega_{\mathscr{F}_E^{\Delta}}, \text{ hence} \\
        \omega_{\mathscr{F}_E^{\Delta}}^{\vee} \otimes \pi^*\omega_{\mathscr{F}} &\rightarrow \mathscr{O}_{\spec E}.
    \end{align}
    In order to see (\ref{eq:regular_foliations_discrepancy}), it suffices to show that the image of $\pi^* \mathscr{K}$ in $\Omega_E^1(\mathrm{log} \, \Delta_E)$ is contained in the dual of $\mathscr{N} := \mathscr{N}_{\mathscr{F}_E^{\Delta}}$.
    Let $\iota : \spec L \rightarrow \spec E$ be the inclusion of the generic point.
    By construction, $\iota^* \pi^* \mathscr{K} = \iota^* \mathscr{N}^{\vee}$ as subsheaves of $\Omega_L^1$, and there is a morphism $\pi^* \mathscr{K} \rightarrow \iota_* \iota^* \mathscr{N}^{\vee}$.
    Since $\mathscr{N}^{\vee}$ is torsion free, $\mathscr{N}^{\vee}$ is a subsheaf of $\iota_* \Omega_L^1$, thus it must be the image of $\pi^* \mathscr{K}$ in $\iota_* \Omega_L^1$.
\end{proof}

\begin{lemma}
    \label{lem:canonical_implies_log_canonical}
    Let $\mathscr{F}$ be a foliation on $X$.
    If $\mathscr{F}$ is canonical at $x$, then $\mathscr{F}$ is log-canonical at $x$.
\end{lemma}

\begin{proof}
    We note that $\mathscr{F}_E^{e} \subseteq \mathscr{F}_E$, hence there is an induced morphism
    \begin{align}
        \label{eq:morphism_canonical_log_canonical}
        \omega_{\mathscr{F}_E^{e}}^{\vee} \rightarrow \omega_{\mathscr{F}_E}^{\vee}.
    \end{align}
    It follows that $\varphi_E^{e}$ factors through $\varphi_E$, thus
    \begin{align}
        \label{eq:discrepancy_inequality}
        a_E(X, \mathscr{F}) \leq a_E^{e}(X, \mathscr{F}).
    \end{align}
    If the left hand side is non-negative, so is the right hand side.
\end{proof}

\begin{lemma}
    \label{lem:canonical_implies_saturated}
    Let $\mathscr{F}$ be a foliation on $(X, \Delta_X)$.
    Assume $\mathscr{F}$ is reflexive.
    If $\mathscr{F}$ is canonical at $x$, then $\mathscr{F}$ is saturated at $x$.
    If $\mathscr{F}$ is log-canonical at $x$ and $x \in \Delta_X$, then $\mathscr{F}$ is log-saturated at $x$.
\end{lemma}

\begin{proof}
    We work locally around $x$.
    We want to show that the log-normal sheaf $\mathscr{N}_{\mathscr{F}}$ is torsion-free.
    Let $\bar{\mathscr{N}}_{\mathscr{F}}$ be the torsion free quotient of $\mathscr{N}_{\mathscr{F}}$ and let
    \begin{align}
        \label{eq:foliation_saturation}
        \bar{\mathscr{F}} = \mathrm{ker} \, \Bigl( \mathscr{T}_X(-\mathrm{log}\,\Delta_X) \rightarrow \bar{\mathscr{N}}_{\mathscr{F}} \Bigr).
    \end{align}
    This is a reflexive foliation on $X$.
    By assumption, $\mathscr{F}$ is also reflexive, hence it suffices to show that $\mathscr{F}\subseteq \bar{\mathscr{F}}$ is an equality at all codimension one points $e \in X$.
    Let $\pi : \spec E \rightarrow X$ be the localisation of $X$ at $e$ and set $\Delta_E$ to be the localisation of $\Delta_X$.
    Note that, if $\mathscr{F}$ is canonical at $x$, $\Delta_E = \emptyset$, and if $\mathscr{F}$ is log-canonical at $x$, $\Delta_E = e$ by assumption.
    Since localisation is flat, it preserves torsion-freeness, and we see that the localisation $\bar{\mathscr{F}}_e$ on $(\spec E, \Delta_E)$ is log-saturated.
    Therefore, it must be equal to $\mathscr{F}_E^{\Delta}$ (Lemma \ref{lem:dvr_log_saturated}).
    We obtain an inclusion $\mathscr{F}_e = \pi^* \mathscr{F} \subseteq \mathscr{F}_E^{\Delta}$, hence a morphism $\omega_{\mathscr{F}_E^{\Delta}} \rightarrow \pi^* \omega_{\mathscr{F}}$.
    Over the generic point $\spec L$, both canonical sheaves restrict to $\omega_{\mathscr{F}_L}$, hence, by applying the morphism $\varphi_E^{\Delta}$ of (\ref{eq:discrepancy_morphism}), we see that the (log-)discrepancy with respect to $E$ is non-positive.
    But $\mathscr{F}$ is (log-)canonical at $x$, thus its (log-)discrepancy must be $0$, i.e. $\omega_{\mathscr{F}_E^{\Delta}} = \pi^* \omega_{\mathscr{F}}$.
    Now, $\pi^* \mathscr{F} \subseteq \mathscr{F}_E^{\Delta}$ is an inclusion of locally free sheaves whose determinant induces an isomorphism, therefore $\pi^* \mathscr{F} = \mathscr{F}_E^{\Delta}$.
    This proves that $\mathscr{F} = \bar{\mathscr{F}}$, i.e. $\mathscr{F}$ is (log-)saturated.
\end{proof}

\begin{remark}
    \label{rem:singularities_summary}
    Let $\mathscr{F}$ be a foliation on $(X, \Delta_X)$.
    We can summarise Lemma \ref{lem:regular_implies_canonical} and Lemma \ref{lem:canonical_implies_log_canonical} with the following diagram of implications.
    \begin{equation}
        \label{diag:singularities_summary}
        \begin{tikzcd}
            \mathscr{F} \text{ regular} \arrow[r] \arrow[d] & \mathscr{F} \text{ canonical} \arrow[r] \arrow[d] & \mathscr{F} \text{ saturated} \arrow[d] \\
            \mathscr{F} \text{ log-regular} \arrow[r] & \mathscr{F} \text{ log-canonical} \arrow[r] & \mathscr{F} \text{ log-saturated}.
        \end{tikzcd}
    \end{equation}
\end{remark}

\begin{lemma}
    \label{lem:regular_canonical_isomorphism}
    Let $\mathscr{F} \subseteq \mathscr{F}^{\prime}$ be two foliations on $(X, \Delta_X)$.
    Assume that $\mathscr{F}$ is reflexive, and that $\mathscr{F}$ and $\mathscr{F}^{\prime}$ have the same rank at the unique generic point generalising $x$.
    If $\mathscr{F}$ is canonical at $x$ and $\mathscr{F}^{\prime}$ is regular at $x$, then $\mathscr{F} = \mathscr{F}^{\prime}$ locally around $x$.
    If $\mathscr{F}$ is log-canonical at $x$, $\mathscr{F}^{\prime}$ is log-regular at $x$, and $x \in \Delta_X$, then $\mathscr{F} = \mathscr{F}^{\prime}$ locally around $x$.
\end{lemma}

\begin{proof}
    We work locally around $x$.
    Consider the induced morphism of log-normal sheaves $\mathscr{N}_{\mathscr{F}} \rightarrow \mathscr{N}_{\mathscr{F}^{\prime}}$.
    By assumption, this is a surjective morphism of sheaves of the same generic rank.
    By Lemma \ref{lem:canonical_implies_saturated}, $\mathscr{N}_{\mathscr{F}}$ is torsion-free, therefore $\mathscr{N}_{\mathscr{F}} = \mathscr{N}_{\mathscr{F}^{\prime}}$ and $\mathscr{F} = \mathscr{F}^{\prime}$.
\end{proof}

\subsection{Ascent and Descent}
\label{subsec:ascent_descent}

We describe the behaviour of (log-)canonical singularities under the operation of pullback foliation.
We show that these singularities are insensitive to smooth morphisms. \par

In this subsection, $f : X \rightarrow Y$ is a morphism of locally Noetherian normal schemes over a field $\mathbbm{k}$ of characteristic zero.
We assume $x \in X$ is a point and set $y := f(x) \in Y$.

\begin{lemma}
    \label{lem:ascent_singularities}
    Let $\mathscr{G}$ be a foliation on $Y$ and suppose that $f$ is smooth.
    Let $\mathscr{F}$ denote the pullback foliation of $\mathscr{G}$ on $X$.
    Then, if $\mathscr{G}$ is (log-)canonical at $y$, so is $\mathscr{F}$ at $x$.
\end{lemma}

\begin{proof}
    By part (4) of Lemma \ref{lem:log_pullback}, it is clear that, if $\mathscr{G}$ is Gorenstein at $y$, so is $\mathscr{F}$ at $x$.
    Since $X$ and $Y$ are normal and we work locally around $x \in X$ and $y \in Y$, we may further assume they are integral with generic points $\eta_X$ and $\eta_Y$ respectively.
    Then $K := \kappa(\eta_X) \subseteq \kappa(\eta_Y) =: L$ is a field extension.
    Let $\spec E \rightarrow X$ be a divisorial valuation over $X$ with centre containing $x$, and let $e$ be its closed point.
    We need to show that it has non-negative (log-)discrepancy. \par

    Define $D = E \cap K \subseteq L$, then $D$ is a either a discrete valuation ring with field of fractions $K$ or is equal to $K$.
    Let $d$ be its closed point.
    By construction, $D$ is a divisorial valuation over $Y$ with centre containing $y$, i.e. $y$ belongs to the scheme-theoretic image of $d$ in $Y$.
    We obtain a commutative diagram
    \begin{equation}
        \label{diag:compare_discrepancies}
        \begin{tikzcd}
            \spec L \arrow[r, "h"] \arrow[d] & \spec K \arrow[d] \\
            \spec E \arrow[r, "g"] \arrow[d, "\pi_E"] & \spec D \arrow[d, "\pi_D"] \\
            X \arrow[r, "f"] & Y.
        \end{tikzcd}
    \end{equation}
    Let $\mathscr{F}|_L$ and $\mathscr{G}|_K$ be the restrictions of $\mathscr{F}$ and $\mathscr{G}$ respectively to the generic points.
    Let $\mathscr{F}_E^{\Delta}$ and $\mathscr{G}_D^{\Delta}$ be the (log-)regular foliations on $(\spec E, \Delta_E)$ and $(\spec D, \Delta_D)$ respectively, which restrict to $\mathscr{F}|_L$ and $\mathscr{G}|_K$ respectively. \par
    
    Firstly, we note that $\mathscr{F}_E^{\Delta}$ is the logarithmic pullback of $\mathscr{G}_D^{\Delta}$.
    Indeed, by Lemma \ref{lem:pullback_foliation_dvr}, the logarithmic pullback of $\mathscr{G}_D^{\Delta}$ is (log-)regular and restricts to $\mathscr{F}|_L$.
    By uniqueness, it must be equal to $\mathscr{F}_E^{\Delta}$. \par

    Next, we note that the lower square of Diagram (\ref{diag:compare_discrepancies}) yields a morphism $\pi_E^* \, \Omega_f^1 \rightarrow \Omega_g^1$.
    We obtain a morphism
    \begin{align}
        \label{eq:base_change_differentials_dvr}
        \omega_g^{\vee} \rightarrow \pi_E^* \, \omega_f^{\vee},
    \end{align}
    using the fact that $\Omega_f^1$ is locally free and that $\Omega_f^1$ and $\Omega_g^1$ have the same generic rank. \par

    Pleasantries aside, we may begin the proof. Observe that there exists a natural morphism
    \begin{align}
        \label{eq:dvr_discrepancy_bottom}
        \omega_{\mathscr{G}_D^{\Delta}}^{\vee} \otimes \pi_D^* \, \omega_{\mathscr{G}} \rightarrow D.
    \end{align}
    Indeed, either $D = K$ and such morphism always exists, or $D$ is a discrete valuation ring and such morphism exists by assumption on the singularity at $y$.
    Now, pulling back by $g$ and applying part (4) of Lemma \ref{lem:log_pullback} to $f$ yields a morphism
    \begin{align}
        \label{eq:dvr_discrepancy_top}
        g^*\omega_{\mathscr{G}_D^{\Delta}}^{\vee} \otimes \pi_E^* \, \omega_f^{\vee} \otimes \pi_E^* \,\omega_{\mathscr{F}} \rightarrow E.
    \end{align}
    Let $\mathscr{C}$ denote the cokernel of the morphism $\mathscr{T}_g \rightarrow \mathscr{F}_E^{\Delta}$ obtained from part (1) of Lemma \ref{lem:log_pullback} applied to $g$.
    Since we are working over a discrete valuation ring, (\ref{eq:relative_foliations_ses}) shows existence of a morphism of locally free sheaves $\mathscr{C} \rightarrow g^* \mathscr{G}_D^{\Delta}$ of the same rank.
    Thus, through (\ref{eq:dvr_discrepancy_top}), we obtain a morphism
    \begin{align}
        \label{eq:dvr_discrepancy_top_i}
        \Lambda^{\mathrm{top}} \, \mathscr{C} \otimes \pi_E^* \, \omega_f^{\vee} \otimes \pi_E^* \, \omega_{\mathscr{F}} \rightarrow E.
    \end{align}
    By construction $\Lambda^{\mathrm{top}} \, \mathscr{C} = \omega_{\mathscr{F}_E^{\Delta}}^{\vee} \otimes \omega_g$, thus we get a morphism
    \begin{align}
        \label{eq:dvr_discrepancy_top_ii}
        \omega_{\mathscr{F}_E^{\Delta}}^{\vee} \otimes \omega_g \otimes \pi_E^* \, \omega_f^{\vee} \otimes \pi_E^* \, \omega_{\mathscr{F}} \rightarrow E.
    \end{align}
    Finally, the morphism in (\ref{eq:base_change_differentials_dvr}) gives the required morphism
    \begin{align}
        \label{eq:dvr_discrepancy_top_iii}
        \omega_{\mathscr{F}_E^{\Delta}}^{\vee} \otimes \pi_E^* \, \omega_{\mathscr{F}} \rightarrow E.
    \end{align}
\end{proof}

\begin{lemma}
    \label{lem:descent_singularities}
    Let $\mathscr{G}$ be a foliation on $Y$ and suppose that $f$ is smooth.
    Let $\mathscr{F}$ denote the pullback foliation of $\mathscr{G}$ on $X$.
    Then, if $\mathscr{F}$ is (log-)canonical at $x$, so is $\mathscr{G}$ at $y$.
\end{lemma}

\begin{proof}
    By part (4) of Lemma \ref{lem:log_pullback}, it is clear that, if $\mathscr{F}$ is Gorenstein at $x$, so is $\mathscr{G}$ at $y$.
    Since $X$ and $Y$ are normal and we work locally around $x \in X$ and $y \in Y$, we may further assume they are integral with generic points $\eta_X$ and $\eta_Y$ respectively.
    Then $K := \kappa(\eta_X) \subseteq \kappa(\eta_Y) =: L$ is a field extension.
    Let $\spec D \rightarrow Y$ be a divisorial valuation over $Y$ with centre containing $y$, and let $d$ be its closed point.
    We need to show that it has non-negative (log-)discrepancy. \par
    
    Let $g : P = X \times_Y \spec D \rightarrow \spec D$ be the base change of $f$ and let $P_d$ be the fibre over $d \in D$.
    Since $g$ is flat, $P_d$ is a Cartier divisor on $P$.
    Furthermore, since scheme-theoretic image commutes with flat base change and $y$ is in the scheme-theoretic image of $d$, we see that $x$ is in the scheme-theoretic image of $P_d$.
    It follows that there must exist a generic point $e \in P_d$ whose scheme-theoretic image in $X$ contains $x$, hence $e$ is a codimension one point in $P$.
    Let $E$ be the localisation of $P$ at $e$, then $E$ is a divisorial valuation over $X$ with centre containing $x$.
    We get a commutative diagram
    \begin{equation}
        \label{eq:fibre_product_discrepancies}
        \begin{tikzcd}
            \spec E \arrow[r] \arrow[rd, "\pi_E"'] & P \arrow[r, "g"] \arrow[d] & \spec D \arrow[d, "\pi_D"] \\
            & X \arrow[r, "f"] & Y.
        \end{tikzcd}
    \end{equation}
    By slight abuse of notation, we also let $g$ denote the morphism from $\spec E$ to $\spec D$ with either logarithmic structures.
    This morphism is a localisation of a smooth morphism, hence all parts of Lemma \ref{lem:log_pullback} still apply. \par
    
    Let $K$ and $L$ be the field of fractions of $D$ and $E$ respectively.
    Let $D^{\prime} := E \cap K \subseteq L$, we prove that $D = D^{\prime}$.
    Certainly, by definition of $D^{\prime}$ as a limit, we have that $D \subseteq D^{\prime}$.
    Also, $D^{\prime}$ is a discrete valuation ring with field of fractions $K$ or is equal to $K$.
    Since $g(e) = d$, we see that $D^{\prime} \neq K$.
    But now $D \subseteq D^{\prime}$ is an inclusion of discrete valuation rings with the same fields of fractions, hence $D = D^{\prime}$. \par
    
    By assumption, there is a natural morphism $\omega_{\mathscr{F}_E^{\Delta}}^{\vee} \otimes \pi_E^* \, \omega_{\mathscr{F}} \rightarrow E$.
    We note that
    \begin{align}
        \label{eq:descent_first_term}
        \omega_{\mathscr{F}_E^{\Delta}}^{\vee} &= g^* \omega_{\mathscr{G}_D^{\Delta}}^{\vee} \otimes \omega_g \text{ and} \\
        \label{eq:descent_second_term}
        \pi_E^* \, \omega_{\mathscr{F}} &= \omega_g^{\vee} \otimes \pi_E^* f^* \omega_{\mathscr{G}}.
    \end{align}
    To see (\ref{eq:descent_first_term}), note that $\mathscr{F}_E^{\Delta}$ is the logarithmic pullback of $\mathscr{G}_D^{\Delta}$ (Lemma \ref{lem:pullback_foliation_dvr}) and $g$ is a localisation of a smooth morphism.
    Thus we may apply part (4) of Lemma \ref{lem:log_pullback} to $g$ in order to obtain the equality.
    To see (\ref{eq:descent_second_term}), we may apply part (4) of Lemma \ref{lem:log_pullback} to $f$ to obtain $\omega_{\mathscr{F}} = \omega_f^{\vee} \otimes f^* \omega_{\mathscr{G}}$, and note that
    \begin{align}
        \label{eq:differentials_base_change_dvr}
        \pi_E^* \, \Omega_{f}^1 =
        \pi_E^* \, \Omega_{X/Y}^1 =
        \left( \Omega_{X/Y}^1|_P \right)_e =
        \left( \Omega_{P/D}^1 \right)_e =
        \Omega_{E/D}^1 = \Omega_{g}^1
    \end{align}
    Taking exterior powers and dualising yields the equality. \par
    
    Combining the two equations yields a morphism
    \begin{align}
        \label{eq:descent_end_result}
        g^* \left( \omega_{\mathscr{G}_D^{\Delta}}^{\vee} \otimes \pi_D^* \, \omega_{\mathscr{G}} \right) \rightarrow E.
    \end{align}
    But now (\ref{eq:descent_end_result}) implies that the natural morphism $\omega_{\mathscr{G}_D^{\Delta}}^{\vee} \otimes \pi_D^* \, \omega_{\mathscr{G}} \rightarrow K$ maps into $g_* \mathscr{O}_{\spec E}= E$.
    Since $D$ is equal to $E \cap K$, we see that its image lies in $D$.
\end{proof}

\subsection{Modifications and Alterations}
\label{subsec:modifications_alterations}

We prove that singularities of foliations are preserved by sheaf-theoretic pullback along modifications and alterations.
The former case is straightforward, however the latter case is only true for log-canonical singularities (see also \cite[Proposition 5.20]{MR1658959} and \cite[Lemma 2.8]{MR4963950}). \par

In this subsection, $f : X \rightarrow Y$ is a morphism of locally Noetherian normal schemes over a field $\mathbbm{k}$ of characteristic zero.
We assume $x \in X$ is a point and set $y := f(x) \in Y$.
Let $\eta_X \in X$ and $\eta_Y \in Y$ be the unique generic points generalising $x$ and $y$.

\begin{lemma}
    \label{lem:log_canonical_birational}
    Let $\mathscr{F}$ be a foliation on $X$ and let $\mathscr{G}$ be a foliation on $Y$.
    Suppose that $f$ induces an isomorphism of generic points $\kappa(\eta_Y) = \kappa(\eta_X)$ and assume there exists a morphism of foliations $f^* \mathscr{G} \rightarrow \mathscr{F}$ which is an isomorphism over $\eta_X$.
    Then, if $\mathscr{G}$ is (log-)canonical at $y$ and $\mathscr{F}$ is Gorenstein at $x$, $\mathscr{F}$ is (log-)canonical at $x$.
\end{lemma}

\begin{proof}
    By assumption $f(\eta_X) = \eta_Y$ and $K := \kappa(\eta_Y) = \kappa(\eta_X) =: L$.
    Let $\pi : \spec E \rightarrow X$ be a divisorial valuation over $X$ with centre containing $x$, and let $e$ be its closed point.
    We need to show that it has non-negative (log-)discrepancy. \par

    By assumption, $\pi$ is a divisorial valuation with centre containing $x$, thus there exists a natural morphism
    \begin{align}
        \label{eq:birational_base_change_bottom}
        \omega_{\mathscr{G}_E^{\Delta}}^{\vee} \otimes \pi^* \omega_{\mathscr{G}} \rightarrow E.
    \end{align}
    Since $f^* \mathscr{G} \rightarrow \mathscr{F}$ is an isomorphism over $K = L$, we have that $\mathscr{F}_E^{\Delta} = \mathscr{G}_E^{\Delta}$, thus it suffices to show that $f$ induces a morphism
    \begin{align}
        \label{eq:canonical_sheaves_gorenstein}
        \omega_{\mathscr{F}} \rightarrow f^* \omega_{\mathscr{G}}.
    \end{align}

    To this end, note that the morphism $f^* \mathscr{G} \rightarrow \mathscr{F}$ induces a morphism
    \begin{align}
        \label{eq:canonical_sheaves_birational}
        \left( \Lambda^{\mathrm{top}} f^* \mathscr{G} \right)^{[1]} \rightarrow \omega^{\vee}_{\mathscr{F}}, 
    \end{align}
    as $\mathscr{F}$ and $\mathscr{G}$ have the same rank.
    At the same time, there exists a natural morphism
    \begin{align}
        \label{eq:canonical_sheaves_birational_helper}
        \left( f^*\omega_{\mathscr{G}} \right)^{\vee} = \left( f^* \left( \Lambda^{\mathrm{top}} \, \mathscr{G}\right)^{\vee} \right)^{\vee} \rightarrow \left( \Lambda^{\mathrm{top}} f^* \mathscr{G} \right)^{[1]}. 
    \end{align}
    Combining (\ref{eq:canonical_sheaves_birational}) with (\ref{eq:canonical_sheaves_birational_helper}), dualising and using the fact that $\mathscr{F}$ and $\mathscr{G}$ are Gorenstein, gives the required morphism.
\end{proof}

\begin{lemma}
    \label{lem:log_canonical_alteration}
    Let $\mathscr{F}$ be a foliation on $X$ and let $\mathscr{G}$ be a foliation on $Y$.
    Suppose that $f$ induces a finite extension $\kappa(\eta_Y) \subseteq \kappa(\eta_X)$ of fields and assume there exists a morphism of foliations $f^* \mathscr{G} \rightarrow \mathscr{F}$ which is an isomorphism over $\eta_X$.
    Then, if $\mathscr{G}$ is log-canonical at $y$ and $\mathscr{F}$ is Gorenstein at $x$, $\mathscr{F}$ is log-canonical at $x$.
\end{lemma}

\begin{proof}
    By assumption $f(\eta_X) = \eta_Y$ and the field extension $K := \kappa(\eta_Y) \subseteq \kappa(\eta_X) =: L$ is finite and étale.
    Let $\pi_E : \spec E \rightarrow X$ be a divisorial valuation over $X$ with centre containing $x$, and let $e$ be its closed point.
    We need to show that it has non-negative log-discrepancy. \par

    Define $D = E \cap K \subseteq L$, then $D$ is either a discrete valuation ring with field of fractions $K$ or is equal to $K$.
    Since $K \subseteq L$ is finite, $D$ must a discrete valuation ring (Lemma \ref{lem:dvr_log_smooth}) and we again obtain Diagram (\ref{diag:compare_discrepancies}).
    Let $d$ be the closed point of $D$.
    Since $\mathscr{G}$ has log-canonical singularities, we have a natural morphism.
    \begin{align}
        \label{eq:dvr_log_discrepancy_bottom}
        \omega_{\mathscr{G}_D^d}^{\vee} \otimes \pi_D^* \, \omega_{\mathscr{G}} \rightarrow D,
    \end{align}
    Pulling back by $g$, yields a morphism
    \begin{align}
        \label{eq:dvr_log_discrepancy_pullback}
        g^*\omega_{\mathscr{G}_D^d}^{\vee} \otimes g^*\pi_D^* \, \omega_{\mathscr{G}} \rightarrow E.
    \end{align}
    By (\ref{eq:canonical_sheaves_gorenstein}), we have a morphism $\omega_{\mathscr{F}} \rightarrow f^* \omega_{\mathscr{G}}$, hence (\ref{eq:dvr_log_discrepancy_pullback}) gives a morphism
    \begin{align}
        \label{eq:dvr_log_discrepancy_factorisation}
        g^*\omega_{\mathscr{G}_D^d}^{\vee} \otimes \pi_E^* \, \omega_{\mathscr{F}} \rightarrow E.
    \end{align}
    Thus, in order to show that the log-discrepancy of $E$ is non-negative, it suffices to show that $g^*\omega_{\mathscr{G}_D^d} = \omega_{\mathscr{F}_E^e}$.
    Since foliations are locally free on a discrete valuation ring, it suffices to show that $g^* \mathscr{G}_D^d = \mathscr{F}_E^e$.
    By employing Lemma \ref{lem:pullback_foliation_dvr}, we see that it suffices to show that $g^* \mathscr{G}_D^d$ is the logarithmic pullback of $\mathscr{G}_D^d$.
    The latter assertion is true because, by Lemma \ref{lem:dvr_log_smooth}, we have that $g^* \mathscr{T}_D(-\mathrm{log} \, d) = \mathscr{T}_E(-\mathrm{log} \, e)$.
\end{proof}

\subsection{Normal Representations}
\label{subsec:normal_representations}

When $x \in \mathcal{X}$ is a point of an algebraic stack $\mathcal{X}$, there is a natural action of the stabiliser group $G_x$ on the tangent space $T_{\mathcal{X},x}$ of $\mathcal{X}$ at $x$.
This can be obtained through an application of the Rim--Schlessinger property (\cite[\S 4]{MR2593684}).
We here show that the same is true for locally free foliations at closed points, i.e there exists a stabiliser Lie algebra $\mathfrak{g}_x$ which acts on the tangent space of the formal stack.
Our approach is considerably more pedestrian. \par

In this subsection, $X$ is a locally Noetherian normal scheme over an algebraically closed field $\mathbbm{k}$ of characteristic zero, and $x \in X$ is a closed point.

\begin{definition}
    \label{def:stabilisers}
    Let $\mathscr{F}$ be a locally free foliation on $X$.
    The \emph{sheaf of formal stabilisers} $\mathscr{P}_{\mathscr{F}}$ of $\mathscr{F}$ is the cokernel of (\ref{eq:regular_locus_morphism}).
    The \emph{stabiliser Lie algebra} $\mathfrak{g}_x$ of $x$ is the dual of the fibre of $\mathscr{P}_{\mathscr{F}}$ over $x$.
    This is a Lie algebra (\cite[Lemma \ref*{reg-lem:stabiliser_surjection}]{bongiorno3}).
\end{definition}

\begin{definition}
    \label{def:singular_subscheme}
    Let $\mathscr{F}$ be a locally free foliation on $X$.
    The \emph{singular subscheme} of $\mathscr{F}$ is the annihilator of the sheaf of formal stabilisers $\mathscr{P}_{\mathscr{F}}$.
    This is an invariant subscheme (\cite[part (2) of Example \ref*{reg-ex:bott_connection}]{bongiorno3}).
\end{definition}

\begin{construction}
    \label{cons:normal_representation}
    Let $\mathscr{F}$ be a locally free foliation on $X$. To start with, assume that $x$ is $\mathscr{F}$-invariant.
    In this case, we have that $\mathscr{F}|_x = \mathfrak{g}_x$ and the tangent space of the formal stack at $x$ is simply the tangent space $T_{X, x}$ of $X$ at $x$. 
    Let $v \in \mathfrak{g}_x$.
    By Nakayama's lemma, locally around $x$, we may find a derivation $\partial_v$ of $\mathscr{F}$ lifting $v$.
    Let $\mathfrak{m} \subseteq \mathscr{O}_{X,x}$ be the maximal ideal of $x$.
    Since $x$ is $\mathscr{F}$-invariant, we have that $\partial_v(\mathfrak{m}) \subseteq \mathfrak{m}$ and we may define the \emph{normal representation} at $x$ to be
    \begin{align}
        \label{eq:normal_representation_definition}
        \nu_x : \mathfrak{g}_x &\rightarrow \mathrm{End}_{\mathbbm{k}} \left( T_{X,x} \right) \\
        v &\rightarrow \left( \partial_v : \frac{\mathfrak{m}}{\mathfrak{m}^2} \rightarrow \frac{\mathfrak{m}}{\mathfrak{m}^2} \right). \nonumber
    \end{align}
    This is independent of the choice of $\partial_v$, for if $\partial$ is a local derivation of $\mathscr{F}$ which restricts to zero over $x$, then $\partial \in \mathfrak{m}\mathscr{F}$ and its induced endomorphism in (\ref{eq:normal_representation_definition}) is zero.
    Furthermore, this is a Lie algebra representation, for if $w \in \mathfrak{g}_x$ is another vector with lift $\partial_w$, we see that $[\partial_v, \partial_w]$ is a lift of $[v, w]$ and, by definition, $[\partial_v, \partial_w] = \partial_v \circ \partial_w - \partial_w \circ \partial_v$ as derivations of $\mathscr{O}_{X,x}$.
    By construction, the normal representation only depends on the formal neighbourhood of $x \in X$ and the completion of $\mathscr{F}$ at $x$. \par
    
    In general, by \cite[Corollary 5.2]{bongiorno3}, we may find a transversal $V \subseteq X$ locally around $x$ such that the restriction $\mathscr{F}|_V$ is locally free, $[X / \mathscr{F}] = [V / \mathscr{F}|_V]$ formal-locally around $x$, and $x$ is $\mathscr{F}|_V$-invariant.
    As a result, we reduce to the case where $x$ is $\mathscr{F}$-invariant.
\end{construction}

\begin{remark}
    \label{rem:differential_normal_representation}
    Let $X \rightarrow \mathcal{X}$ be a minimal presentation of an algebraic stack with affine diagonal and generically finite stabilisers.
    Let $\mathscr{F}$ denote the induced locally free foliation.
    It follows that the closed point $x$ is $\mathscr{F}$-invariant.
    Let $G$ denote its stabiliser group and let $\mathfrak{g}$ denote its stabiliser Lie algebra.
    By \cite[Lemma \ref*{reg-lem:stabiliser_surjection}]{bongiorno3}, $\mathfrak{g}$ is the Lie algebra of $G$. \par

    Now consider both normal representations at $x$ given by
    \begin{align}
        \label{eq:normal_representation_group}
        \rho_x : G &\rightarrow \mathrm{GL}_{\mathbbm{k}} \left( T_{X, x} \right) \text{ and} \\
        \label{eq:normal_representation_lie_algebra}
        \nu_x : \mathfrak{g} &\rightarrow \mathrm{End}_{\mathbbm{k}} \left( T_{X, x} \right).
    \end{align}
    We explain why (\ref{eq:normal_representation_lie_algebra}) is the differential of (\ref{eq:normal_representation_group}).
    Let $\hat{\mathcal{X}}$ be the infinitesimal stack associated to $\mathscr{F}$.
    Since $\mathscr{F}$ is defined as the foliation induced by $X \rightarrow \mathcal{X}$, we have that $\hat{\mathcal{X}} \rightarrow \mathcal{X}$ is formally étale, and $B\mathfrak{g}$ is the preimage of $BG$.
    But a formally étale morphism induces an isomorphism of tangent spaces, thus we deduce that the differential of $\rho_x$ is $\nu_x$.
\end{remark}

\subsection{Semistable morphisms}
\label{subsec:semistable_morphisms}

In this subsection, $X$ and $Y$ are locally Noetherian normal schemes over an algebraically closed field $\mathbbm{k}$ of characteristic zero, and $\Delta_X \subseteq X$ and $\Delta_Y \subseteq Y$ are reduced Cartier divisors on $X$ and $Y$ respectively.
We let $f : (X, \Delta_X) \rightarrow (Y, \Delta_Y)$ be a logarithmic morphism over $\mathbbm{k}$.
We assume $x \in X$ is a closed point and set $y := f(x) \in Y$.

\begin{definition}
    \label{def:semistable_morphism}
    A morphism $f : (X, \Delta_X) \rightarrow (Y, \Delta_Y)$ is semistable at $x$ if 
    \begin{enumerate}
        \item $X$ is smooth $x$ and $(X, \Delta_X)$ is logarithmically smooth at $x$,
        \item $Y$ is smooth $y$ and $(Y, \Delta_Y)$ is logarithmically smooth at $y$, and
        \item $f$ is integral, saturated and logarithmically smooth at $x$.
    \end{enumerate}
\end{definition}

\begin{remark}
    \label{rem:semistable_coordinates}
    If $f$ is semistable at $x$, then formal-locally around $x$ and $y$, we may find systems of local parameters $\langle x_1,\ldots,x_n \rangle$ and $\langle y_1,\ldots,y_m \rangle$ such that
    \begin{align}
        \label{eq:semistable_coordinates}
        \hat{f}^{\#} : \hat{\mathscr{O}}_{Y, y} &\rightarrow \hat{\mathscr{O}}_{X, x} \\
        y_1 &\rightarrow x_1 x_ 2 \ldots x_{k_1} \nonumber \\
        y_2 &\rightarrow x_{k_1 + 1} x_{k_1 + 2} \ldots x_{k_2} \nonumber \\
        \vdots &\quad\quad\quad\quad\quad \vdots \nonumber \\
        y_m &\rightarrow x_{k_{m-1} + 1} x_{k_{m-1} + 2} \ldots x_{k_m}, \nonumber 
    \end{align}
    for natural numbers $1 \leq k_1 < k_2 < \ldots < k_m \leq n$.
    See \cite[Theorem 3.5]{MR1463703} for more details.
\end{remark}

\begin{proposition}
    \label{prop:semistable_representation}
    Assume that $f$ is a semistable morphism at $x$ and let $\mathscr{T}_f$ denote the relative logarithmic tangent sheaf.
    Then $\mathscr{T}_f$ is log-regular at $x$, the stabiliser Lie algebra of $x$ is Abelian and its normal representation is faithful and semisimple.
\end{proposition}

\begin{proof}
    We work locally around $x$ and use the notation of Remark \ref{rem:semistable_coordinates}. \par
    
    By definition of logarithmic smoothness, $\mathscr{T}_f$ is locally free and the morphism
    \begin{align}
        \label{eq:semistable_regular}
        \Omega_X^1(\mathrm{log}\,\Delta_X) \rightarrow \mathscr{T}_f^{\vee}
    \end{align}
    is surjective.
    This shows that $\mathscr{T}_f$ is log-regular. \par

    Now, using properties of locally free sheaves of differentials, we see that the completion of $\mathscr{T}_f$ at $x$ is the logarithmic tangent sheaf of the completion of $f$ at $x$.
    Since the normal representation only depends on the formal neighbourhood, we may assume that $f$ is given by the monomials in (\ref{eq:semistable_coordinates}).

    In this setting, a computation shows that $\mathscr{T}_f$ is generated by the vector fields
    \begin{align}
        \label{eq:semistable_vector_fields_i}
        \partial_i := x_i \partial_{x_i} - {x_{i+1}} \partial_{x_{i+1}} &\text{ for $k_{j-1} < i < k_j$ and $1 \leq j \leq m$, and} \\
        \label{eq:semistable_vector_fields_ii}
        \partial_l := \partial_{x_l} \quad\quad\quad\quad\quad\,\,\,\,\, &\text{ for $k_m < l \leq n$,}
    \end{align}
    with the convention $k_0 = 0$.
    A transversal $V \subseteq X$ is given by $x_{k_m +1} = \ldots = x_n = 0$.
    Indeed, the ideal $(x_1,\ldots,x_{k_m})$ is $\mathscr{T}_f$-invariant and intersects $V$ transversally.
    Thus, after slicing, we may assume $k_m = n$ and $x$ is $\mathscr{T}_f$-invariant. \par

    It is clear that the Lie bracket of the vector fields in (\ref{eq:semistable_vector_fields_i}) is zero, thus the stabiliser Lie algebra $\mathfrak{g}_x$ is Abelian. \par
    
    It is equally clear that $\nu_x$ is semisimple, for each vector field acts diagonally with one weight equal to $1$, one weight equal to $-1$, and all other weights being trivial.
    Furthermore, $\nu_x$ is faithful.
    Indeed, let $i_1$ be the smallest natural number for which $\partial_{i_1}$ is defined.
    Then a linear combination of vector fields which acts as zero must map the class $\langle x_{i_1} \rangle$ to zero.
    But $\partial_{i_1}$ is the only vector field acting non-trivially on $\langle x_{i_1} \rangle$, thus its coefficient is zero.
    Applying induction yields the result.
\end{proof}
    \section{Singularities of Foliations on Stacks}
\label{sec:stacks}

We move on to algebraic stacks.
We first treat the case of foliations on Deligne--Mumford stacks (\S \ref{subsec:foliations_dm_stacks}) and then the case of algebraic stacks (\ref{subsec:relative_singularities}).
We then study the behaviour of singularities under base change (\S \ref{subsec:singularities_base_change}).
We only prove results which will be directly relevant for us.

\subsection{Foliations on Deligne--Mumford Stacks}
\label{subsec:foliations_dm_stacks}

We show that the definition of an algebraic foliation on a scheme effortlessly extends to the case of Deligne--Mumford stacks.
The advantages of considering foliations in such generality were firstly discussed in \cite{MR4890447}. \par

In this subsection, $\mathbf{X}$ and $\mathbf{Y}$ are locally Noetherian normal Deligne--Mumford stacks over a field $\mathbbm{k}$ of characteristic zero, and $\Delta_{\mathbf{X}} \subseteq \mathbf{X}$ and $\Delta_{\mathbf{Y}} \subseteq \mathbf{Y}$ are reduced Cartier divisors on $\mathbf{X}$ and $\mathbf{Y}$ respectively.
We let $f : (\mathbf{X}, \Delta_{\mathbf{X}}) \rightarrow (\mathbf{Y}, \Delta_{\mathbf{Y}})$ be a logarithmic morphism over $\mathbbm{k}$.
We assume $x \in \mathbf{X}$ is a point and set $y := f(x) \in \mathbf{Y}$.

\begin{remark}
    \label{rem:foliation_dm_stack}
    Suppose that $X \rightarrow \mathbf{X}$ is an étale presentation of $\mathbf{X}$ and let
    \begin{equation}
        \label{eq:dm_groupoid_presentation}
        \begin{tikzcd}
            R := X \times_{\mathbf{X}} X \arrow[r, shift left, "s"] \arrow[r, shift right, "t"'] &  X
        \end{tikzcd}
    \end{equation}
    be the associated étale groupoid.

    \begin{enumerate}[itemsep=1em]
        \item A Cartier divisor $\Delta_{\mathbf{X}}$ on $\mathbf{X}$ is an $R$-invariant Cartier divisor $\Delta_X$ on $X$.
        Recall that $\Delta_X$ is $R$-invariant if $s^{-1} (\Delta_X) =: \Delta_R := t^{-1} (\Delta_X)$.
        Being reduced does not depend on the étale presentation.

        \item A coherent sheaf on $\mathbf{X}$ is a coherent sheaf $\mathscr{F}$ on $X$ together with a choice of isomorphism $s^* \mathscr{F} \cong t^* \mathscr{F}$ satisfying the identity and cocycle conditions.
        A coherent sheaf on $\mathbf{X}$ is locally free (\emph{resp.} torsion-free) whenever $\mathscr{F}$ is so on $X$.
        These properties do not depend on the choice of presentation.
        
        \item The sheaf of differentials $\Omega_{\mathbf{X}}^1(\mathrm{log} \, \Delta_{\mathbf{X}})$ is the coherent sheaf $\Omega_X^1(\mathrm{log} \, \Delta_X)$ together with the isomorphism
        \begin{align}
            \label{eq:dm_differentials_linearisation}
            s^* \Omega_X^1(\mathrm{log} \, \Delta_X) = \Omega_R^1(\mathrm{log} \, \Delta_R) = t^*\Omega_X^1(\mathrm{log} \, \Delta_X),
        \end{align}
        coming from the fact that both $s$ and $t$ are étale.
        The tangent sheaf $\mathscr{T}_{\mathbf{X}}(-\mathrm{log} \, \Delta_{\mathbf{X}})$ is its dual.
        
        \item If $f$ is a logarithmic morphism and $Y \rightarrow {\mathbf{Y}}$ is an étale presentation, we may define the sheaf of logarithmic differentials $\Omega_f^1$ on the Deligne--Mumford stack $\mathbf{X} \times_{\mathbf{Y}} X$ by choosing a presentation $X \rightarrow \mathbf{X} \times_{\mathbf{Y}} X$ and repeating part (3) above.
        By étale descent, this defines a sheaf of logarithmic differentials $\Omega_f^1$ on $\mathbf{X}$.
        We let $\mathscr{T}_f$ denote its coherent dual.
    \end{enumerate}
\end{remark}

\begin{remark}
    \label{rem:log_tangent_sheaf_stack}
    Through Remark \ref{rem:foliation_dm_stack}, we can use Remark \ref{rem:log_tangent_sheaf} to deduce that a local derivation $\partial$ of $\mathscr{T}_{\mathbf{X}}$ is in $\mathscr{T}_f$ if and only if it is a derivation on $\mathbf{X}$ relative to $\mathbf{Y}$ under which $\Delta_{\mathbf{X}}$ is invariant.
\end{remark}

\begin{definition}
    \label{def:all_together_dm_stack}
    With Remark \ref{rem:foliation_dm_stack} in mind, we can define \emph{foliations} on logarithmic Deligne--Mumford stacks, \emph{rank}, \emph{foliation invariance} of logarithmic substacks, \emph{log-normal sheaves}, the properties of being \emph{log-saturated}, \emph{log-regular} and \emph{Gorenstein}, as well as \emph{stabiliser Lie algebras} and \emph{normal representations}.
\end{definition}

\begin{example}
    \label{ex:foliations_dm_stacks}
    \begin{enumerate}[itemsep=1em]
        \item Suppose that $X$ is a scheme and let $\mathscr{F}$ be a foliation on $X$.
        Let $G$ be a finite group acting on $X$ compatibly with $\mathscr{F}$.
        This means that the natural $G$-linearisation on $\mathscr{T}_X$ restricts to its subsheaf $\mathscr{F}$.
        Then $\mathscr{F}$ is a foliation on the Deligne--Mumford stack $\mathbf{X} := [X/G]$.
        \item Let $\mathscr{F}$ be a $\mathbb{Q}$-Gorenstein foliation on a scheme $X$ of index $m$.
        We let $\tilde{X}$ be the Gorenstein cover of $\mathscr{F}$, so that the pullback foliation $\tilde{\mathscr{F}}$ is Gorenstein.
        By construction, the cyclic group $\mu_m$ acts on $\tilde{X}$ compatibly with $\tilde{\mathscr{F}}$, and we get a foliation on the Deligne--Mumford stack $\tilde{\mathbf{X}} := [\tilde{X}/\mu_m]$.
        The induced morphism $\tilde{\mathbf{X}} \rightarrow X$ is a birational modification.
    \end{enumerate}
\end{example}

\begin{lemma}
    \label{lem:regular_foliations_isomorphic_stacks}
    Let $\mathscr{F} \subseteq \mathscr{F}^{\prime}$ be two foliations on $(\mathbf{X}, \Delta_{\mathbf{X}})$.
    Assume that $\mathscr{F}$ and $\mathscr{F}^{\prime}$ have the same rank at the unique generic point generalising $x$.
    If $\mathscr{F}$ is log-regular at $x$, then $\mathscr{F} = \mathscr{F}^{\prime}$ locally around $x$.
\end{lemma}

\begin{proof}
    We can check whether a morphism of sheaves is an isomorphism on an étale presentation $X \rightarrow \mathbf{X}$.
    This holds by Lemma \ref{lem:regular_foliations_isomorphic}.
\end{proof}

\begin{definition}
    \label{def:singularities_dm_stacks}
    Let $\mathscr{F}$ be a foliation on $(\mathbf{X}, \Delta_{\mathbf{X}})$.
    Then $\mathscr{F}$ is \emph{(log-)canonical} at $x$ if it is so on an étale presentation at a point representing $x$.
    By Lemma \ref{lem:ascent_singularities} and Lemma \ref{lem:descent_singularities}, this is independent of both the presentation and the point chosen. 
\end{definition}

\begin{remark}
    \label{rem:foliated_triples_q}
    We can extend the definition of discrepancies and singularities of foliations to $\mathbb{Q}$-Gorenstein foliations on schemes with effective $\mathbb{Q}$-Cartier divisor with coefficients in $[0, 1]$.
    Indeed, by part (2) of Example \ref{ex:foliations_dm_stacks}, we may find a natural Deligne--Mumford stack such that both foliation and boundary divisor are Cartier.
    We can then use Definition \ref{def:singularities_dm_stacks}.
\end{remark}

\begin{lemma}
    \label{lem:regular_canonical_isomorphism_stacks}
    Let $\mathscr{F} \subseteq \mathscr{F}^{\prime}$ be two foliations on $(\mathbf{X}, \Delta_{\mathbf{X}})$.
    Assume that $\mathscr{F}$ is reflexive, and that $\mathscr{F}$ and $\mathscr{F}^{\prime}$ have the same rank at the unique generic point generalising $x$.
    If $\mathscr{F}$ is canonical at $x$ and $\mathscr{F}^{\prime}$ is regular at $x$, then $\mathscr{F} = \mathscr{F}^{\prime}$ locally around $x$.
    If $\mathscr{F}$ is log-canonical at $x$, $\mathscr{F}^{\prime}$ is log-regular at $x$, and $x \in \Delta_X$, then $\mathscr{F} = \mathscr{F}^{\prime}$ locally around $x$.
\end{lemma}

\begin{proof}
    We can check whether a morphism of sheaves is an isomorphism on an étale presentation $X \rightarrow \mathbf{X}$.
    This holds by Lemma \ref{lem:regular_canonical_isomorphism}.
\end{proof}

\begin{definition}
    \label{def:semistable_morphism_stacks}
    A morphism $f : (\mathbf{X}, \Delta_{\mathbf{X}}) \rightarrow (\mathbf{Y}, \Delta_{\mathbf{Y}})$ is semistable at $x$ if there exist étale presentations $Y \rightarrow \mathbf{Y}$ and $X \rightarrow \mathbf{X} \times_{\mathbf{Y}} Y$ such that the composition $X \rightarrow Y$ is semistable at a representative of $x$ as a morphism of schemes (see Definition \ref{def:semistable_morphism}).
\end{definition}

\begin{proposition}
    \label{prop:semistable_representation_stacks}
    Assume that $f$ is a semistable morphism at $x$ and let $\mathscr{T}_f$ denote the relative logarithmic tangent sheaf.
    Then $\mathscr{T}_f$ is log-regular at $x$, the stabiliser Lie algebra of $x$ is Abelian and its normal representation is faithful and semisimple.
\end{proposition}

\begin{proof}
    By definition of semistability, we may find étale presentations $X \rightarrow \mathbf{X}$ and $Y \rightarrow \mathbf{Y}$, and a semistable morphism $X \rightarrow Y$.
    We then know the thesis holds on $X$ by Proposition \ref{prop:semistable_representation}.
    But now the étale presentation $X \rightarrow \mathbf{X}$ induces an isomorphism of stabiliser Lie algebras and normal representations at $x$.
\end{proof}

\subsection{Relative Singularities of Algebraic Stacks}
\label{subsec:relative_singularities}

We are now ready to define relative singularities of algebraic stacks. \par

In this subsection, $f : \mathcal{X} \rightarrow \mathcal{Y}$ is a morphism of locally Noetherian normal algebraic stacks over a field $\mathbbm{k}$ of characteristic zero.
We assume $x \in \mathcal{X}$ is a point and set $y = f(x) \in \mathcal{Y}$.

\begin{definition}
    \label{def:relative_singularities}
    $\mathcal{X}$ is \emph{relatively regular} (resp. \emph{relatively (log-)canonical}) at $x$ if there exists a smooth presentation $X \rightarrow \mathcal{X}$ such that the induced foliation is regular (resp. (log-)canonical) at a point in the preimage of $x$.
    By Lemma \ref{lem:regular_foliation_dm_stack} (resp. Lemma \ref{lem:ascent_singularities} and Lemma \ref{lem:descent_singularities}), this is independent of both the presentation and the point chosen. 
\end{definition}

The implications in the top row of Diagram (\ref{diag:singularities_summary}) continue to hold true in this setting.

\begin{lemma}
    \label{lem:regular_foliation_dm_stack}
    A point $x \in \mathcal{X}$ is relatively regular if and only if its stabiliser group is finite.
\end{lemma}

\begin{proof}
    Let $X \rightarrow \mathcal{X}$ be a smooth presentation with associated groupoid $R$ and induced foliation $\mathscr{F}$.
    Let $G$ be the stabiliser group of a closed point $\tilde{x} \in X$ representing $x$.
    The stabiliser Lie algebra $\mathfrak{g}$ of $\tilde{x}$ is the Lie algebra of $G$ (Remark \ref{rem:differential_normal_representation}).
    By definition, $\mathscr{F}$ is regular at $\tilde{x}$ if and only if $\mathfrak{g} = 0$, i.e. if and only if $G$ is finite.
\end{proof}

\begin{lemma}
    \label{lem:regular_locus_pullback}
    Suppose that $f$ is representable by Deligne--Mumford stacks.
    Then, if $\mathcal{Y}$ is a relatively regular at $y$, so is $\mathcal{X}$ at $x$.
\end{lemma}

\begin{proof}
    If the stabiliser of $y$ is finite, the assumption on $f$ implies that so is the stabiliser of $x$.
    We conclude by Lemma \ref{lem:regular_foliation_dm_stack}.
\end{proof}

\subsection{Singularities and Base Change}
\label{subsec:singularities_base_change}

In this subsection, $f : \mathcal{X} \rightarrow \mathcal{Y}$ is a morphism of locally Noetherian normal algebraic stacks over a field $\mathbbm{k}$ of characteristic zero.
We assume $x \in \mathcal{X}$ is a point and set $y = f(x) \in \mathcal{Y}$.

\begin{lemma}
    \label{lem:singularities_representable_morphisms}
    Suppose that $f$ is a representable étale morphism.
    Then $\mathcal{X}$ is a relatively (log-)canonical at $x$ if and only if so is $\mathcal{Y}$ at $y$.
\end{lemma}

\begin{proof}
    Let $f_Y : X \rightarrow Y$ be the base change of $f$ by a smooth presentation $Y \rightarrow \mathcal{Y}$, and let $\mathscr{F}$ and $\mathscr{G}$ be the induced locally free foliations on $X$ and $Y$ respectively.
    By assumption, $\mathscr{F} = f^* \mathscr{G}$.
    Since $f_Y$ is étale, $\mathscr{F}$ is the logarithmic pullback of $\mathscr{G}$ and we conclude by Lemma \ref{lem:ascent_singularities} and Lemma \ref{lem:descent_singularities}.
\end{proof}

\begin{lemma}
    \label{lem:log_canonical_alteration_stacks}
    Suppose that $f$ is representable by an alteration of Deligne--Mumford stacks.
    Then, if $\mathcal{Y}$ is a relatively (log-)canonical at $y$, so is $\mathcal{X}$ at $x$.
\end{lemma}

\begin{proof}
    This follows from Lemma \ref{lem:log_canonical_alteration} upon choosing presentations between which the induced morphism is necessarily surjective and generically finite étale.
\end{proof}
    \section{Resolution of Singularities}
\label{sec:resolution}

In this section, we want to resolve the relative singularities of the morphism $X \rightarrow \mathcal{X}$ (Remark \ref{rem:foliation_is_morphism}) in order to obtain log-regular singularities.
The main tool is functorial semistable reduction (\cite{abramovich2020}), for we have seen in Proposition \ref{prop:semistable_representation_stacks} that semistable morphisms yield log-regular singularities.
The plan of this section is outlined in the introduction (Main Ideas, Step (I)-(IV)). \par

In this section, $\mathcal{X}$ is a normal quasi-separated algebraic stack of finite type over an algebraically closed field $\mathbbm{k}$ of characteristic zero with finite generic stabilisers.
Let $X \rightarrow \mathcal{X}$ be a smooth presentation and let $\mathscr{F}$ denote the induced locally free foliation on $X$.
Let $x \in \mathcal{X}$ be a closed point.

\subsection{Construction of Rational Map}
\label{subsec:rational_map}

The aim of this subsection is to prove the following.

\begin{lemma}
    \label{lem:rational_map}
    There exists a dense open smooth Deligne--Mumford substack $\mathbf{U^{\circ}} \subseteq \mathcal{X}$ which is a gerbe over a smooth affine scheme $B^{\circ}$ of finite type over $\mathbbm{k}$.
\end{lemma}

\begin{remark}
    \label{rem:gomez_mont}
    If such open substack is represented by $U^{\circ} \subseteq X$, we get a diagram
    \begin{equation}
        \label{diag:rational_map}
        \begin{tikzcd}
            U^{\circ} \arrow[r] \arrow[d] & \mathbf{U^{\circ}} \arrow[r] \arrow[d] & B^{\circ} \\
            X \arrow[r] & \mathcal{X} &
        \end{tikzcd}
    \end{equation}
    Thus, the lemma proves existence of a rational map $X \dashedrightarrow B^{\circ}$ inducing $\mathscr{F}$ on $U^{\circ}$.
    This result is analogous to the construction of an invariant rational map associated to an algebraically integrable foliation carried out in \cite[Theorem 3]{MR1017286}.
\end{remark}

\begin{lemma}
    \label{lem:smooth_gerbe}
    There exists a dense open smooth Deligne--Mumford substack $\mathbf{U} \subseteq \mathcal{X}$ which is a gerbe over an algebraic space.
\end{lemma}

\begin{proof}
    Any algebraic stack $\mathcal{X}$ contains a maximal open smooth substack.
    Since $\mathcal{X}$ is normal, such substack is dense.
    Thus, up to replacing $\mathcal{X}$, we may assume $\mathcal{X}$ is smooth. \par

    Similarly, an algebraic stack contains a maximal open Deligne--Mumford substack (\cite[\href{https://stacks.math.columbia.edu/tag/0DSM}{Lemma 0DSM}]{stacks-project}).
    By assumption, this is dense, for the generic stabilisers of $\mathcal{X}$ are trivial.
    Thus, up to replacing $\mathcal{X}$, we may assume $\mathcal{X}$ is smooth and Deligne--Mumford. \par

    Finally, since $\mathcal{X}$ is normal and quasi-separated, there exists a dense open substack $\mathbf{U}$ which is a gerbe over an algebraic space (\cite[\href{https://stacks.math.columbia.edu/tag/06RC}{Proposition 06RC}]{stacks-project}).
    Note that $\mathbf{U}$ must be smooth and Deligne--Mumford, since so is $\mathcal{X}$.
\end{proof}

\begin{lemma}
    \label{lem:coarse_moduli_space}
    Let $\mathcal{X}$ be a smooth quasi-separated stack of finite type over $\mathbbm{k}$.
    Assume $\mathcal{X} \rightarrow B$ is a gerbe over an algebraic space $B$.
    Then $\mathcal{X} \rightarrow B$ is smooth and $B$ is smooth, quasi-separated and of finite type over $\mathbbm{k}$.
\end{lemma}

\begin{proof}
    Recall that any gerbe is smooth (\cite[\href{https://stacks.math.columbia.edu/tag/0DN8}{Lemma 0DN8}]{stacks-project}) and being smooth descends along smooth surjective morphisms, thus $B$ is smooth over $\mathbbm{k}$.
    Furthermore, as the diagonal of $\mathcal{X}$ is quasi-compact by assumption, so is the diagonal of $B$ (\cite[\href{https://stacks.math.columbia.edu/tag/0DQL}{Lemma 0DQL}]{stacks-project}).
    As a result, $B$ is quasi-separated.
    In order to show that $B$ is of finite type, it suffices to show that $B$ is quasi-compact.
    This follows from \cite[\href{https://stacks.math.columbia.edu/tag/050X}{Lemma 050X}]{stacks-project}.
\end{proof}

\begin{proof}
    [Proof of Lemma \ref{lem:rational_map}]
    Let $\mathbf{U}$ be the dense open substack constructed in Lemma \ref{lem:smooth_gerbe}.
    $\mathbf{U}$ is a smooth Deligne--Mumford stack which is a gerbe over an algebraic space $B$.
    By Lemma \ref{lem:coarse_moduli_space}, $B$ is quasi-separated, hence it contains an open dense subset which is a scheme (\cite[\href{https://stacks.math.columbia.edu/tag/06NH}{Proposition 06NH}]{stacks-project}).
    Such subscheme must be smooth, quasi-separated and of finite type over $\mathbbm{k}$, for so is $B$.
    Thus \cite[\href{https://stacks.math.columbia.edu/tag/01ZX}{Lemma 01ZX}]{stacks-project} implies that there exists a dense open affine subscheme $B^{\circ} \subseteq B$.
    It follows that $B^{\circ}$ is a smooth affine scheme of finite type over $\mathbbm{k}$.
    Let $\mathbf{U^{\circ}}$ be its preimage of $B^{\circ}$.
    Since a gerbe is smooth, the preimage of $B^{\circ}$ is still dense in $\mathbf{U}$, which is dense in $\mathcal{X}$.
\end{proof}

\subsection{Resolution of Indeterminacy Locus}
\label{subsec:indeterminacy_locus}

Let $B^{\circ} \subseteq B$ be a smooth projective compactification and let $f : \mathcal{X} \supset \mathbf{U^{\circ}} \rightarrow B^{\circ} \subseteq B$ denote the induced rational map.
We resolve the indeterminacy locus of $f$.

\begin{lemma}
    \label{lem:indeterminacy_locus}
    The indeterminacy locus of $f : \mathcal{X} \supset \mathbf{U^{\circ}} \rightarrow B^{\circ} \subseteq B$ can be resolved by a representable logarithmic modification
    \begin{equation}
        \label{diag:indeterminacy_locus}
        \begin{tikzcd}
            & (\mathcal{Z}, \Delta_{\mathcal{Z}}) \arrow[dl, "\sigma"'] \arrow[dr, "g"] & \\
            \mathcal{X} \arrow[rr, "f", dashed] && B,
        \end{tikzcd}
    \end{equation}
    such that $\Delta_{\mathcal{Z}} = \mathcal{Z} \setminus \mathbf{U^{\circ}}$, $\sigma$ is an isomorphism over $\mathbf{U^{\circ}}$ and $g$ is smooth on $\mathbf{U^{\circ}}$.
    Furthermore $\mathcal{Z}$ is a locally Noetherian normal quasi-separated algebraic stack of finite type over $\mathbbm{k}$ with finite generic stabilisers.
\end{lemma}

\begin{proof}
    Up to blowing up $\mathcal{X} \setminus \mathbf{U^{\circ}}$ with its reduced structure, we may assume that the complement of $\mathbf{U^{\circ}} \subseteq \mathcal{X}$ is a reduced Cartier divisor $\Delta_{\mathcal{X}} \subseteq \mathcal{X}$ represented by $\Delta_X \subseteq X$.
    The morphisms $U^{\circ} \rightarrow B$ and $\mathbf{U^{\circ}} \rightarrow B$ induce morphisms $U^{\circ} \rightarrow X \times B$ and $\mathbf{U^{\circ}} \rightarrow \mathcal{X} \times B$ respectively.
    Let $Z$ and $\mathcal{Z}$ be the respective images.
    We obtain a commutative diagram
    \begin{equation}
        \label{diag:images_resolution}
        \begin{tikzcd}
            U^{\circ} \arrow[r] \arrow[d] & Z \arrow[r] \arrow[d] & X \times B \arrow[r] \arrow[d] & X \arrow[d] \\
            \mathbf{U^{\circ}} \arrow[r] & \mathcal{Z} \arrow[r] & \mathcal{X} \times B \arrow[r] & \mathcal{X},
        \end{tikzcd}
    \end{equation}
    where every square is Cartesian.
    This follows from the fact that scheme-theoretic image commutes with flat base change.
    Therefore we obtain a representable morphism of stacks $\sigma : \mathcal{Z} \rightarrow \mathcal{X}$, which is an isomorphism on the smooth open substack $\mathbf{U^{\circ}}$, and is projective, since so is $B$.
    Thus, by Lemma \ref{lem:rational_map}, $g$ is smooth on $\mathbf{U^{\circ}}$.
    Up to replacing $\mathcal{Z}$ with its normalisation, we may assume $\mathcal{Z}$ is normal.
    The remaining properties of $\mathcal{Z}$ are easily derived from the corresponding properties of $\mathcal{X}$.
    Let $\Delta_{\mathcal{Z}} = \sigma^{-1}(\Delta_{\mathcal{X}})$, then, by construction, $\mathbf{U^{\circ}} = \mathcal{Z} \setminus \Delta_{\mathcal{Z}}$.
\end{proof}

\begin{proposition}
    \label{prop:canonical_invariant_valuation}
    With the notation of Lemma \ref{lem:indeterminacy_locus}, suppose further that $\mathcal{X}$ is canonical at $x$.
    Then $g(\Delta_{\mathcal{Z}})$ is not dense in $B$.
    In particular, there exists an open subset $W \subseteq B$ such that $g_W : g^{-1}(W) \rightarrow W$ is a gerbe.
\end{proposition}

\begin{proof}
    We work locally around $x$ and we may get rid of prime components of $\Delta_{\mathcal{Z}}$ whose centre does not contain $x$.
    Furthermore, since $\Delta_{\mathcal{Z}}$ is a finite union of prime components, in order to prove the proposition, we may assume it is irreducible.
    Let $e$ be the codimension one point of $Z$ contained in $\Delta_Z$ and let $E$ be the localisation of $Z$ at $e$.
    Consider the composition $g_E : E \rightarrow B$ and its relative tangent sheaf $\mathscr{T}_{g_E}$. \par

    We first show that $e \in E$ is $\mathscr{T}_{g_E}$-invariant.
    By construction, $\mathscr{T}_{g_E}$ is equal to $\sigma_X^* \mathscr{F}$ over the generic point of $E$, thus $\mathscr{T}_{g_E} \subseteq \mathscr{F}_E$.
    On the other hand, since $\sigma(e) = x$, $\sigma_X^* \mathscr{F}$ is a canonical foliation on $E$ at $e$ (Lemma \ref{lem:log_canonical_birational}).
    Since $\mathscr{F}_E$ is a regular foliation on $E$, we have that $\sigma_X^* \mathscr{F} = \mathscr{F}_E$ (Lemma \ref{lem:regular_canonical_isomorphism}).
    It follows that $\mathscr{T}_{g_E} \subseteq \sigma_X^* \mathscr{F}$, thus, since $e$ is $\sigma_X^* \mathscr{F}$-invariant, it is also $\mathscr{T}_{g_E}$-invariant. \par

    Suppose for contradiction that $g_E(e) = \eta$, where $\eta$ is the generic point of $B$.
    The generic point of $E$ is also mapped to $\eta \in B$, thus, since $E$ is reduced, the scheme-theoretic image of $g_E$ must be $\spec \kappa(\eta) \in B$.
    This implies that $\mathscr{T}_{g_E} = \mathscr{T}_{E/\kappa(\eta)}$.
    But now $E$ is a discrete valuation ring over $\kappa(\eta)$ such that $e$ is $\mathscr{T}_{E/\kappa(\eta)}$-invariant.
    This cannot happen.
    Indeed, after completion, $E \cong k_E\llbracket t \rrbracket$ for a local parameter $t \in E$ and a coefficient field $k_E \subseteq E$ necessarily containing $\kappa(\eta)$.
    Thus, we may construct a derivation mapping $t$ to $1$, contradicting invariance of $e$. \par

    Now, let $W \subseteq B^{\circ}$ be an open dense subset which does not intersect the image of $\Delta_{\mathcal{Z}}$.
    By construction, $\mathbf{U^{\circ}} = \mathcal{Z} \setminus \Delta_{\mathcal{Z}}$, thus $g^{-1} (W) \subseteq \mathbf{U^{\circ}}$.
    It follows that, since $\mathbf{U^{\circ}} \rightarrow B^{\circ}$ is a gerbe (Lemma \ref{lem:smooth_gerbe}), so is its base change $g_W : g^{-1} (W) \rightarrow W$.
\end{proof}

\subsection{Semistable Reduction}
\label{subsec:semistable_reduction}

We now verify that we can apply the main results of \cite{abramovich2020} on functorial semistable reduction and carry out the procedure.

\begin{lemma}
    \label{lem:semistable_reduction}
    Let $g : (\mathcal{Z}, \Delta_{\mathcal{Z}}) \rightarrow B$ be the morphism of Lemma \ref{lem:indeterminacy_locus}.
    Then, there exists a commutative diagram
    \begin{equation}
        \label{eq:semistable_morphism}
        \begin{tikzcd}
            (\mathcal{Z}^{\mathrm{ss}}, \Delta_{\mathcal{Z}}^{\mathrm{ss}}) \arrow[r, "g^{\mathrm{ss}}"] \arrow[d] & ({\mathbf{B}}^{\mathrm{ss}}, \Delta_{\mathbf{B}}^{\mathrm{ss}}) \arrow[d] \\
            (\mathcal{Z}, \Delta_{\mathcal{Z}}) \arrow[r, "g"] & B,
        \end{tikzcd}
    \end{equation}
    where $({\mathbf{B}}^{\mathrm{ss}}, \Delta_{\mathbf{B}}^{\mathrm{ss}}) \rightarrow B$ is a logarithmic alteration, $(\mathcal{Z}^{\mathrm{ss}}, \Delta_{\mathcal{Z}}^{\mathrm{ss}}) \rightarrow (\mathcal{Z}, \Delta_{\mathcal{Z}})$ is representable by a logarithmic alteration $({\mathbf{Z}}^{\mathrm{ss}}, \Delta_{\mathbf{Z}}^{\mathrm{ss}}) \rightarrow (Z, \Delta_Z)$ of Deligne--Mumford stacks, and $g_{\mathbf{Z}}^{\mathrm{ss}} : ({\mathbf{Z}}^{\mathrm{ss}}, \Delta_{\mathbf{Z}}^{\mathrm{ss}}) \rightarrow ({\mathbf{B}}^{\mathrm{ss}}, \Delta_{\mathbf{B}}^{\mathrm{ss}})$ is semistable.
\end{lemma}

\begin{proof}
    We want to apply functorial semistable reduction to the morphism $g_Z : (Z, \Delta_Z) \rightarrow B$ and its smooth groupoid.
    Let $R := Z \times_{\mathcal{Z}} Z \rightrightarrows Z$ be the associated smooth groupoid.
    Let also $s^{-1}(\Delta_Z) =: \Delta_R := t^{-1}(\Delta_Z)$.
    By construction, there exists a morphism $\mathcal{Z} \rightarrow B$, thus the two compositions $R \rightrightarrows Z \rightarrow B$ are equal.
    Let $g_Z^{\prime} : (R, \Delta_R) \rightarrow B$ denote the composition. \par
    
    By Lemma \ref{lem:thm_hypotheses}, the hypotheses of \cite[Theorem 1.2.17]{abramovich2020} hold, thus we obtain a commutative diagram
    \begin{equation}
        \label{diag:semistable_reduction}
        \begin{tikzcd}
            ({\mathbf{R}}^{\mathrm{ss}}, \Delta_{\mathbf{R}}^{\mathrm{ss}}) \arrow[r, shift left] \arrow[r, shift right] \arrow[d] & ({\mathbf{Z}}^{\mathrm{ss}}, \Delta_{\mathbf{Z}}^{\mathrm{ss}}) \arrow[r, "g_{\mathbf{Z}}^{\mathrm{ss}}"] \arrow[d] & ({\mathbf{B}}^{\mathrm{ss}}, \Delta_{\mathbf{B}}^{\mathrm{ss}}) \arrow[d] \\
            (R, \Delta_R) \arrow[r, shift left] \arrow[r, shift right] & (Z, \Delta_Z) \arrow[r, "g_Z"] & B,
        \end{tikzcd}
    \end{equation}
    where $({\mathbf{B}}^{\mathrm{ss}}, \Delta_{\mathbf{B}}^{\mathrm{ss}}) \rightarrow B$ is a logarithmic alteration, $({\mathbf{Z}}^{\mathrm{ss}}, \Delta_{\mathbf{Z}}^{\mathrm{ss}}) \rightarrow (Z, \Delta_Z)$ is a logarithmic alteration of Deligne--Mumford stacks and $g_{\mathbf{Z}}^{\mathrm{ss}}$ is semistable.
    Furthermore, by functoriality, $({\mathbf{R}}^{\mathrm{ss}}, \Delta_{\mathbf{R}}^{\mathrm{ss}})$ is the saturated base change of $(R, \Delta_R) \times_B ({\mathbf{B}}^{\mathrm{ss}}, \Delta_{\mathbf{B}}^{\mathrm{ss}})$, hence the left-most squares of Diagram (\ref{diag:semistable_reduction}) are Cartesian with either source or target morphism.
    In fact, we do not need to saturate, as both source and target are smooth, as opposed to merely logarithmically smooth.
    It follows that $s^{-1}(\Delta_{\mathbf{Z}}^{\mathrm{ss}}) = \Delta_{\mathbf{R}}^{\mathrm{ss}} = t^{-1}(\Delta_{\mathbf{Z}}^{\mathrm{ss}})$.
    Applying functoriality again to the inverse and composition of $R$ shows that ${\mathbf{R}}^{\mathrm{ss}}$ is a smooth groupoid on ${\mathbf{Z}}^{\mathrm{ss}}$.
    We conclude by letting $\mathcal{Z}^{\mathrm{ss}}$ denote the associated stack $[{\mathbf{Z}}^{\mathrm{ss}} / {\mathbf{R}}^{\mathrm{ss}}]$, and $\Delta_{\mathcal{Z}}^{\mathrm{ss}}$ denote the induced divisor.
\end{proof}

\begin{lemma}
    \label{lem:thm_hypotheses}
    In the setting of Lemma \ref{lem:semistable_reduction}, the hypotheses of \cite[Theorem 1.2.17]{abramovich2020} hold.
\end{lemma}

\begin{proof}
    $g_Z$ is a generically logarithmically regular morphism of quasi-excellent Noetherian logarithmic Deligne--Mumford stacks of characteristic zero.
    This is true, for $R$, $Z$ and $B$ are of finite type over $\mathbbm{k}$ and the restriction of $g_Z$ to $U^{\circ} \subseteq Z$ is smooth. \par

    The morphisms $(R, \Delta_R) \rightrightarrows (Z, \Delta_Z)$ are logarithmically regular and surjective.
    By hypothesis, both source and target are smooth surjective morphisms. \par
    
    $g_Z$ and $g_Z^{\prime}$ are locally embeddable into a logarithmic orbifold.
    This follows from the fact that both $g_Z$ and $g_Z^{\prime}$ are of finite type (\cite[Corollary 8.4.5]{abramovich2020}).
\end{proof}

\subsection{Formal Representability of Resolution}
\label{subsec:formal_representability}

We show that, when $x$ is a log-canonical singularity, the resolution of singularities is formally representable.

\begin{proposition}
    \label{prop:formal_representable}
    There exists a morphism of algebraic stacks $\pi : (\mathcal{Z}^{\mathrm{ss}}, \Delta_{\mathcal{Z}}^{\mathrm{ss}}) \rightarrow \mathcal{X}$, which is representable by a logarithmic alteration $\pi_X : ({\mathbf{Z}}^{\mathrm{ss}}, \Delta_{\mathbf{Z}}^{\mathrm{ss}}) \rightarrow X$, and a semistable morphism $g^{\mathrm{ss}} := g_{\mathbf{Z}}^{\mathrm{ss}} : ({\mathbf{Z}}^{\mathrm{ss}}, \Delta_{\mathbf{Z}}^{\mathrm{ss}}) \rightarrow {\mathbf{B}}^{\mathrm{ss}}$, such that $\pi_X^* \mathscr{F}$ is a foliation on $({\mathbf{Z}}^{\mathrm{ss}}, \Delta_{\mathbf{Z}}^{\mathrm{ss}})$ contained in $\mathscr{T}_{g^{\mathrm{ss}}}$.
    If $\mathcal{X}$ is relatively log-canonical at $x$ and $z \in \mathcal{Z}^{\mathrm{ss}}$ satisfies $\pi(z) = x$, then $\mathscr{F} = \mathscr{T}_{g^{\mathrm{ss}}}$ locally around $z$.
\end{proposition}

\begin{proof}
    We first obtain a rational map $f : \mathcal{X} \supseteq \mathbf{U^{\circ}} \rightarrow B^{\circ} \subseteq B$ by Lemma \ref{lem:rational_map}.
    We then resolve the indeterminacy locus by Lemma \ref{lem:indeterminacy_locus}, and apply Lemma \ref{lem:semistable_reduction} to obtain $g^{\mathrm{ss}}$.
    By combining Diagram (\ref{diag:semistable_reduction}) and Diagram (\ref{diag:images_resolution}), we see that there is a sequence of alterations fitting in the following commutative diagram
    \begin{equation}
        \label{eq:total_resolution}
        \begin{tikzcd}
            {\mathbf{Z}}^{\mathrm{ss}} \arrow[r, "\tau_Z"] \arrow[d] & Z \arrow[r, "\sigma_X"] \arrow[d] & X \arrow[d] \\
            \mathcal{Z}^{\mathrm{ss}} \arrow[r, "\tau"] & \mathcal{Z} \arrow[r, "\sigma"] & \mathcal{X},
        \end{tikzcd}
    \end{equation}
    where every square is Cartesian.
    Let $\pi = \sigma \circ \tau$ and $\pi_X = \sigma_X \circ \tau_Z$, and let $\mathscr{F}^{\mathrm{ss}}$ denote the locally free foliation on ${\mathbf{Z}}^{\mathrm{ss}}$ associated to the smooth morphism ${\mathbf{Z}}^{\mathrm{ss}} \rightarrow \mathcal{Z}^{\mathrm{ss}}$.
    It follows that $\pi_X^* \mathscr{F} = \mathscr{F}^{\mathrm{ss}}$. \par

    By construction, $g^{\mathrm{ss}}$ is an $\mathscr{F}^{\mathrm{ss}}$-invariant morphism and $\Delta_Z$ is $\mathscr{F}^{\mathrm{ss}}$-invariant.
    Thus, the local sections of $\mathscr{F}^{\mathrm{ss}}$ consist of derivations on ${\mathbf{Z}}^{\mathrm{ss}}$ relative to ${\mathbf{B}}^{\mathrm{ss}}$ under which $\Delta_{\mathbf{Z}}^{\mathrm{ss}}$ is invariant.
    It follows from Remark \ref{rem:log_tangent_sheaf_stack} that
    \begin{align}
        \label{eq:inclusion_foliations}
        \pi_X^*\mathscr{F} = \mathscr{F}^{\mathrm{ss}} \subseteq \mathscr{T}_{g^{\mathrm{ss}}}
    \end{align}
    as subsheaves of $\mathscr{T}_{{\mathbf{Z}}^{\mathrm{ss}}}$ of the same generic rank, for $\pi_X$ is an alteration. \par

    We now show that, if $\mathcal{X}$ is relatively log-canonical at $x$, then (\ref{eq:inclusion_foliations}) is an equality in a neighbourhood of $z$.
    We work locally around $z$.
    Recall that, by Proposition \ref{prop:semistable_representation_stacks}, $\mathscr{T}_{g^{\mathrm{ss}}}$ is log-regular on $({\mathbf{Z}}^{\mathrm{ss}}, \Delta_{\mathbf{Z}}^{\mathrm{ss}})$. \par

    Suppose first that $z \notin \Delta_{\mathcal{Z}}^{\mathrm{ss}}$.
    Observe that $x = \pi(z) \in \mathbf{U^{\circ}}$.
    Indeed, since $\sigma^{-1}(\mathcal{X} \setminus \mathbf{U^{\circ}}) = \Delta_{\mathcal{Z}}$ and $\tau$ is a logarithmic morphism, $\mathcal{Z}^{\mathrm{ss}} \setminus \pi^{-1}(\mathbf{U^{\circ}}) = \pi^{-1}(\mathcal{X} \setminus \mathbf{U^{\circ}}) = \tau^{-1}(\Delta_{\mathcal{Z}}) \subseteq \Delta_{\mathcal{Z}}^{\mathrm{ss}}$.
    But now $\mathcal{X}$ is relatively regular at $x$ (Lemma \ref{lem:regular_foliation_dm_stack}), hence $\mathcal{Z}^{\mathrm{ss}}$ is relatively regular at $z$ (Lemma \ref{lem:regular_locus_pullback}), so that $\mathscr{F}^{\mathrm{ss}}$ is regular on ${\mathbf{Z}}^{\mathrm{ss}}$ at a representative of $z$.
    We conclude by Lemma \ref{lem:regular_foliations_isomorphic_stacks}. \par
    
    Now suppose $z \in \Delta_{\mathcal{Z}}^{\mathrm{ss}}$.
    By Lemma \ref{lem:log_canonical_alteration_stacks}, $\mathcal{Z}^{\mathrm{ss}}$ is relatively log-canonical at $z$.
    It follows that $\mathscr{F}^{\mathrm{ss}}$ is log-canonical at a representative $\tilde{z} \in \mathbf{Z}^{\mathrm{ss}}$ of $z$.
    But now, (\ref{eq:inclusion_foliations}) is an inclusion of a log-canonical foliation into a log-regular foliation of the same generic rank, and $\tilde{z} \in \Delta_{\mathbf{Z}}^{\mathrm{ss}}$.
    Equality follows from Lemma \ref{lem:regular_canonical_isomorphism_stacks}.
\end{proof}
    \section{Proof of Main Theorem}
\label{sec:main_theorem}

We are ready to prove the Main Theorem and its corollaries.
We do so in \S \ref{subsec:main_log_canonical} and \S \ref{subsec:main_canonical}.
In \S \ref{subsec:auxiliary}, we show some auxiliary results concerning algebraic groups, which are needed for the proof.
In particular, we show a simple version of Borel's fixed point theorem for Deligne--Mumford stacks. \par

In this section, $\mathcal{X}$ is a normal algebraic stack of finite type over an algebraically closed field $\mathbbm{k}$ of characteristic zero with affine diagonal and finite generic stabilisers.
Let $x \in \mathcal{X}(\mathbbm{k})$ be a point and let $G$ be its stabiliser.
Let $X \rightarrow \mathcal{X}$ be a minimal presentation such that the representative $\tilde{x} \in X$ is an invariant closed point, and let $\mathscr{F}$ denote the induced locally free foliation on $X$.

\subsection{Log-canonical Singularities}
\label{subsec:main_log_canonical}

We assume $\mathcal{X}$ is relatively log-canonical at $x$.

\begin{proof}
    [Proof of part (1) of Main Theorem]
    We work locally around $x \in \mathcal{X}$.
    By assumption, $G$ is an affine algebraic group over $\mathbbm{k}$.
    Let $G^{\circ} \subseteq G$ be the connected component of the identity of $G$.
    By Lemma \ref{lem:additive_subgroups_torus}, it suffices to show that there exists no additive subgroup $\mathbb{G}_a \subseteq G^{\circ}$.
    Assume for contradiction that an additive subgroup exists. \par
    
    Using the fact that $\mathcal{X}$ is log-canonical at $x$, we apply Proposition \ref{prop:formal_representable} to find a morphism of algebraic stacks $\pi : (\mathcal{Z}^{\mathrm{ss}}, \Delta_{\mathcal{Z}}^{\mathrm{ss}}) \rightarrow \mathcal{X}$ representable by a logarithmic alteration $\pi_X : ({\mathbf{Z}}^{\mathrm{ss}}, \Delta_{\mathbf{Z}}^{\mathrm{ss}}) \rightarrow X$, and a semistable morphism $g^{\mathrm{ss}} : ({\mathbf{Z}}^{\mathrm{ss}}, \Delta_{\mathbf{Z}}^{\mathrm{ss}}) \rightarrow {\mathbf{B}}^{\mathrm{ss}}$ such that $\pi_X^* \mathscr{F} = \mathscr{T}_{g^{\mathrm{ss}}}$. \par
    
    Define the algebraic substack $\mathcal{V} = \pi^{-1}(x) \subseteq \mathcal{Z}^{\mathrm{ss}}$.
    By construction, it is represented by a closed immersion $\mathbf{V} \subseteq {\mathbf{Z}}^{\mathrm{ss}}$.
    Since $G$ fixes $x$, $\mathcal{V} = [\mathbf{V}/G]$.
    Since $\mathbb{G}_a \subseteq G$ also acts on $\mathbf{V}$, there exists a point $v \in \mathbf{V}(\mathbbm{k})$ whose stabiliser $H_v$ contains a copy of the additive subgroup (Lemma \ref{lem:additive_group_stack}).
    Using the normal representation of the algebraic stack $\mathcal{Z}^{\mathrm{ss}}$ at the image of $v$, we see that $\mathbb{G}_a \subseteq H_v$ acts on the tangent space $T_{\mathcal{Z}^{\mathrm{ss}}, v}$. \par

    Consider the foliation $\pi_X^* \mathscr{F}$ and let $\mathfrak{h}_{v}$ be the stabiliser Lie algebra of $v$.
    $\mathfrak{h}_{v}$ acts on the tangent space $T_{\mathcal{Z}^{\mathrm{ss}}, v}$ via the normal representation.
    We know that $\mathfrak{h}_v$ is the Lie algebra of $H_v$, and the normal representation of $\mathfrak{h}_v$ is the differential of the normal representation of $H_v$ (Remark \ref{rem:differential_normal_representation}).
    Differentiating the action of the additive subgroup $\mathbb{G}_a \subseteq H_v$ on $T_{\mathcal{Z}^{\mathrm{ss}}, v}$ yields an endomorphism $\partial$ of $T_{\mathcal{Z}^{\mathrm{ss}}, v}$.
    Crucially, $\partial$ is nilpotent, for it is induced by an action of the additive group, and it follows that the normal representation of $\pi_X^* \mathscr{F}$ at $v$ contains a nilpotent endomorphism. \par
    
    On the other hand, the normal representation of $\mathscr{T}_{g^{\mathrm{ss}}}$ at $v$ is faithful and semisimple (Proposition \ref{prop:semistable_representation_stacks}).
    Since $\pi_X^* \mathscr{F} = \mathscr{T}_{g^{\mathrm{ss}}}$, being semisimple implies that $\partial$ acts trivially, and being faithful gives the desired contradiction. \par

    But now $G$ is reductive, and, by \cite[Theorem 1.1]{MR4088350}, there exists an étale affine neighbourhood $\mathcal{U} \rightarrow \mathcal{X}$ of $x$ such that $\mathcal{U}$ is a $G$-quotient stack.
    In particular, $\mathcal{U}$ admits a good moduli space.
\end{proof}

\begin{proof}
    [Proof of Corollary A]
    We can immediately apply part (1) of the Main Theorem to the quotient stack $[X/G]$ around the image of $x$.
\end{proof}

\subsection{Canonical Singularities}
\label{subsec:main_canonical}

We assume $\mathcal{X}$ is relatively canonical at $x$.

\begin{proof}
    [Proof of part (2) of Main Theorem]
    Since $\mathcal{X}$ is relatively canonical at $x$, $\mathcal{X}$ is also relatively log-canonical at $x$ (Lemma \ref{lem:canonical_implies_log_canonical}), and we apply part (1) of the Main Theorem to find an étale affine neighbourhood $\mathcal{U} \rightarrow \mathcal{X}$ of $x$ with good moduli space $q : \mathcal{U} \rightarrow Q$.
    Since canonical singularities ascend along representable étale morphisms (Lemma \ref{lem:singularities_representable_morphisms}), we may replace $\mathcal{U}$ with $\mathcal{X}$ and assume $q : \mathcal{X} \rightarrow Q$ is a good moduli space.
    We want to show there exists a point $x^{\prime} \in \mathcal{X}$ such that $q^{-1}(q(x^{\prime})) = \{x^{\prime}\}$. \par

    By Lemma \ref{lem:indeterminacy_locus}, we get a proper morphism $\sigma : (\mathcal{Z}, \Delta_{\mathcal{Z}}) \rightarrow \mathcal{X}$ representable by a modification, and a morphism $g : (\mathcal{Z}, \Delta_{\mathcal{Z}}) \rightarrow B$ to a smooth projective scheme $B$.
    Since $\mathcal{X}$ is canonical at $x$, by Proposition \ref{prop:canonical_invariant_valuation}, there exists an open dense subset $W \subseteq B$ such that $g_W : g^{-1}(W) \rightarrow B$ is a gerbe.
    In particular, $g_W$ is a homeomorphism (\cite[\href{https://stacks.math.columbia.edu/tag/06R9}{Lemma 06R9}]{stacks-project}).
    As a result, if $w \in W$ is a closed point, we see that there exists a unique closed point $z^{\prime} \in g^{-1} (w) \subseteq \mathcal{Z}^{\prime}$.
    But $\sigma$ is a proper morphism, hence $x^{\prime} := \sigma(z^{\prime})$ is a closed point of $\mathcal{X}$.
    Since $q$ is a good moduli space, it induces a bijection between closed points of $\mathcal{X}$ and closed point of $Q$.
\end{proof}

\begin{proof}
    [Proof of Corollary B]
    Let $\mathcal{U} \rightarrow \mathcal{X}$ be an affine étale neighbourhood represented by $U \rightarrow X$ with good moduli space $q : \mathcal{U} \rightarrow Q$.
    The universal property of good moduli spaces induces a morphism of coherent sheaves $\mathscr{F}|_U \rightarrow \mathscr{T}_{U/Q}$.
    It suffices to show that they have the same generic rank.
    Let $\mathcal{U}^{\mathrm{s}}$ be the non-empty (part (2) of Main Theorem) locus of stable points of $\mathcal{U}$.
    By \cite[Proposition 2.6]{MR4255045}, this is an open substack.
    Since $q$ induces a homeomorphism between $\mathcal{U}^{\mathrm{s}}$ and a subset of $Q$, it follows that $\mathrm{dim} \, \mathcal{U}^{\mathrm{s}} = \mathrm{dim} \, Q$.
    Now, $\mathrm{dim} \, \mathcal{U}^{\mathrm{s}} = \mathrm{dim} \, \mathcal{U} = \mathrm{dim} \, U - \mathrm{rk} \, \mathscr{F}$, and thus
    \begin{align}
        \label{eq:dimension_formula}
        \mathrm{dim} \, Q = \mathrm{dim} \, U - \mathrm{rk} \, \mathscr{F}.
    \end{align}
    But the generic rank of $\mathscr{T}_{U/Q}$ is $\mathrm{dim} \, U - \mathrm{dim} \, Q$, which is equal to $\mathrm{rk} \, \mathscr{F}$.
\end{proof}

\subsection{Auxiliary Results}
\label{subsec:auxiliary}

In this subsection, we work over an algebraically closed field $\mathbbm{k}$ of characteristic zero.

\begin{lemma}
    \label{lem:additive_group_stack}
    Let $\mathbf{V}$ be a proper Deligne--Mumford stack over $\mathbbm{k}$ and let the additive group $\mathbb{G}_a$ act on $\mathbf{V}$.
    Then there exists a point $v \in \mathbf{V}(\mathbbm{k})$ fixed by $\mathbb{G}_a$.
\end{lemma}

\begin{proof}
    By the Keel--Mori theorem, $\mathbf{V}$ admits a coarse moduli space $q : \mathbf{V} \rightarrow V$ as a proper algebraic space (\cite[Theorem 1.1]{MR1432041}).
    Since $V$ is a uniform categorical quotient, the $\mathbb{G}_a$ action on $\mathbf{V}$ descends to a $\mathbb{G}_a$ action on $V$.
    By Lemma \ref{lem:additive_group_space}, there exists a closed point $v \in V$ fixed by $\mathbb{G}_a$.
    But $q$ induces a bijection between closed points of $\mathbf{V}$ and $V$, thus $q^{-1}(v) \subseteq \mathbf{V}(\mathbbm{k})$ is a singleton fixed by $\mathbb{G}_a$.
\end{proof}

\begin{lemma}
    \label{lem:additive_group_space}
    Let $V$ be a proper algebraic space over $\mathbbm{k}$ and let the additive group $\mathbb{G}_a$ act on $V$.
    Then there exists a closed point $v \in V$ fixed by $\mathbb{G}_a$.
\end{lemma}

\begin{proof}
    Pick a closed point $v \in V$ whose orbit has minimal dimension.
    We first show that its orbit $O$ is closed.
    Let $\bar{O} \subseteq V$ be the set-theoretic closure of the orbit morphism $\sigma_v : \mathbb{G}_a \rightarrow V$ of $v$.
    Note that $\bar{O}$ is $\mathbb{G}_a$-invariant.
    Indeed, $\mathbb{G}_a$ is dense in $\bar{O}$, thus $\mathbb{G}_a \times \mathbb{G}_a$ is dense in $\mathbb{G}_a \times \bar{O}$, so that the image of $\mathbb{G}_a \times \bar{O} \rightarrow V$ is $\bar{O}$.
    If the orbit were not closed, there would exist a closed point $v^{\prime} \in V$ not contained in the image of $\sigma_v$ whose orbit would necessarily have strictly lower dimension than the orbit of $v$.
    Now, suppose $v$ is not fixed, then its stabiliser is trivial, for $\mathbb{G}_a$ does not have non-trivial algebraic subgroups.
    It follows that $O$ is a $\mathbb{G}_a$-torsor, hence $O$ is affine.
    On the other hand, $O$ is a closed subspace of a proper algebraic space, hence is proper.
    We deduce that $O$ is finite, thus $O = \{v\}$, for $\mathbb{G}_a$ is connected.
\end{proof}
    
\begin{lemma}
    \label{lem:additive_subgroups_torus}
    Let $G$ be an affine connected algebraic group over $\mathbbm{k}$.
    Suppose there exists no additive subgroup $\mathbb{G}_a \subseteq G$, then $G$ is an algebraic torus.
\end{lemma}

\begin{proof}
    If the unipotent radical of $G$ were not trivial, it would necessarily contain a copy of $\mathbb{G}_a$.
    It follows that $G$ is reductive, thus its derived subgroup $G^{\prime}$ is semisimple.
    If $G^{\prime}$ were not trivial, a Borel subgroup would necessarily contain a copy of $\mathbb{G}_a$.
    It follows that $G$ is Abelian, reductive and connected, i.e. an algebraic torus.
\end{proof}

    \bibliographystyle{alpha}
    \bibliography{extra/bibliography}

\end{document}